\documentclass[12pt]{amsart}

\usepackage{amsmath, amsthm, amssymb, amsfonts} 
\usepackage{mathtools} 
\usepackage{bm} 
\usepackage{graphicx} 
\usepackage{hyperref} 
\usepackage{cleveref} 
\usepackage{geometry} 
\usepackage{enumitem} 
\usepackage{booktabs} 
\usepackage{siunitx} 
\usepackage{microtype} 
\usepackage{xcolor}
\usepackage{a4wide}
\usepackage{cancel}
\usepackage{ulem}
\allowdisplaybreaks
\allowdisplaybreaks[4]

\theoremstyle{plain}
\newtheorem{theorem}{Theorem}[section]
\newtheorem{lemma}[theorem]{Lemma}

\numberwithin{equation}{section}

\theoremstyle{definition}

\newtheorem{remark}[theorem]{Remark}

\newtheorem*{claim}{Claim}

\title[Gradient growth on the sphere]{Growth of vorticity gradient for the Euler equation on the sphere}

\author{Daomin Cao, Junhong Fan, Guolin Qin}
\address{State Key Laboratory of Mathematical Sciences, Academy of Mathematics and Systems Science, Chinese Academy of Sciences, Beijing 100190, P.R. China and University of Chinese Academy of Sciences, Beijing 100049,  P.R. China}
\email{dmcao@amt.ac.cn}

\address{Institute of Applied Mathematics, AMSS, Chinese Academy of Sciences, Beijing 100190, and University of Chinese Academy of Sciences, Beijing 100049,  P.R. China}
\email{fanjunhong@amss.ac.cn}

\address{State Key Laboratory of Mathematical Sciences, Academy of Mathematics and Systems Science, Chinese Academy of Sciences, Beijing 100190, P.R. China}
\email{qinguolin18@mails.ucas.ac.cn}

\begin{document}

\begin{abstract}
   We prove that for solutions of the Euler equation on the sphere, the vorticity gradient can grow at most double-exponentially in time, and we show that this upper bound is sharp by constructing explicit solutions with odd symmetry that exhibit double-exponential growth in the hemisphere. We also extend the results to the case of a rotating sphere. This seems to be the first result on the growth of the  vorticity gradient for ideal fluids on a compact manifold with non‑trivial geometry.
\end{abstract}

 \maketitle{\small{\bf Keywords:}   Incompressible Euler equation, rotating sphere, double-exponential growth.  \\

\section{Introduction}
\label{sec:introduction}
In this paper, we consider an incompressible ideal fluid on the unit sphere $\mathbb{S}^2:=\{\mathbf{x}=(x_1,x_2,x_3)\in\mathbb{R}^3:\,x_1^2+x_2^2+x_3^2=1\}$. The motion of the fluid can be described by the following incompressible Euler equation:
\begin{equation}\label{Euler}
	 	\begin{cases}
	 		\partial_t\mathbf{v}+\nabla_\mathbf{v}\mathbf{v}=-\nabla P,\ \ &\text{on}\ \  \mathbb{S}^2\times \mathbb{R}^+,\\
	 		\text{div}\,\mathbf{v}=0,\ \ &\text{on}\ \ \mathbb{S}^2\times \mathbb{R}^+,

	 	\end{cases}
	 \end{equation} 
     where $ \mathbf{v}(\mathbf{x},t)=(v_1(\mathbf{x},t),v_2(\mathbf{x},t),v_3(\mathbf{x},t)) $ is the velocity field belongs to $T_\mathbf{x}\mathbb{S}^2$, the tangent plane of $\mathbb{S}^2$ at $\mathbf{x}$, $ P$ is the scalar pressure. $\nabla_\mathbf{v}$ denotes the covariant derivative along $\mathbf{v}$, $\nabla$ the gradient operator, and $\text{div}$ the divergence operator—all taken with respect to the standard metric on $\mathbb{S}^2$. Global well‑posedness of \eqref{Euler} in suitable function spaces can be established by adapting the classical methods of \cite{MB,MCP}, in close analogy with the two‑dimensional Euclidean case. For instance, a detailed proof of global well‑posedness in $H^s$ for any $s>2$ is given in \cite{TM}. We also note that a rotating sphere (about some axis) can be considered, which is physically relevant; see \cite{CAG}. 

    In this paper, we focus on the problem of vorticity gradient growth for the Euler equation on the sphere. This problem has long been studied for planar domains and the torus, but, as we shall show, remains completely unexplored on the sphere. Before presenting our main results, we briefly recall the known background.

\subsection{Background}

For the two-dimensional incompressible Euler equation, the maximum vorticity magnitude $\|\omega\|_{L^\infty}$ is conserved in time. However, its gradient $\|\nabla\omega\|_{L^\infty}$ can grow, reflecting the continuous generation of small-scale features, which is a key aspect of two-dimensional fluid dynamics. Quantifying the possible growth rate of $\|\nabla\omega\|_{L^\infty}$ has therefore been a longstanding challenge, dating back to the foundational works of Wolibner \cite{WW}, Hölder \cite{HE} and Yudovich \cite{YV3}. Indeed, they proved that the gradient growth does not exceed a double exponential. More precisely: 
\begin{equation*} \|\nabla\omega(\cdot,t)\|_{L^\infty}\leq(1+\|\nabla\omega_0\|_{L^\infty})^{C\exp{(Ct)}},
\end{equation*}
where the constant $C$ depends only on $\|\omega_0\|_{L^\infty}$. For more details on this double-exponential rate, the reader is referred to \cite{KV,YV1}. The validity of the double-exponential upper bound is known only for domains with regular boundaries. For domains whose boundary is smooth except at two points, finite-time blow-up can occur, as shown by Kiselev and Zlatoš \cite{KAZA}. The question of whether this bound is sharp, namely whether there exist solutions whose vorticity gradient indeed grows at a double‑exponential rate, has since become a central open problem.

Over the past decades, a series of works have been devoted to understanding the extent to which vorticity gradients can grow in two-dimensional Euler flows. These investigations have spanned a variety of domain geometries, each offering distinct dynamical mechanisms and technical challenges. In bounded domains with impermeable boundaries, Nadirashvili \cite{NS} and Yudovich \cite{JV,YV2} demonstrated the possibility of unbounded and at least linear growth, respectively. For more related results in this direction, the reader is referred to \cite{DE,DEJ,KA,KAA}. A major breakthrough came from Kiselev and Šverák \cite{KV}, who constructed smooth initial vorticity on a disk, motivated by numerical evidence of singularity formation in three dimensions obtained by Luo and Hou \cite{LH}, and proved that the double‑exponential upper bound is attainable for all time. Their construction is based on the stability of a particular odd-symmetric steady state (the Bahouri–Chemin vortex) and the presence of a boundary, where the flow compresses vorticity gradients at an accelerating rate. This result was later extended to general bounded  planar domains with an axis of symmetry by Xu \cite{XX}. Additionally, it is noteworthy that the double-exponential growth result in \cite{KV} has recently been extended to the free boundary setting by Hu, Luo and Yao \cite{HLY}, where the fluid domain is allowed to deform and the Biot-Savart law is not directly available.

On the torus $\mathbb{T}^2$, which has no boundary but is compact, progress has also been achieved. Denisov \cite{DS1} established the existence of infinite-time superlinear growth of the vorticity gradient, and later showed that double‑exponential growth is possible over arbitrarily long finite time intervals \cite{DS2}. In a related direction, he also constructed patch solutions to the 2D Euler equation with a prescribed regular stirring, for which the distance between two approaching patches decreases double-exponentially in time \cite{DSA}. Zlatoš \cite{ZA1} considered perturbations of the steady state $\omega^*(\mathbf{x})= \operatorname{sgn}(x_1)\operatorname{sgn}(x_2)$ and proved exponential growth of $\|\nabla\omega(\cdot,t)\|_{L^\infty}$ for initial data of class $C^{1,\alpha}$ with $\alpha<1$, as well as exponential growth of $\|\nabla^2\omega(\cdot,t)\|_{L^\infty}$ for smooth vorticities. In this setting, the lack of boundary is compensated by the $C^{1,\alpha}$ regularity, which ensures the presence of more vorticity near the hyperbolic stagnation point.

The situation on unbounded domains such as the whole plane $\mathbb{R}^2$ has remained far more delicate. In contrast to the double‑exponential growth achieved on a disk \cite{KV}, the best-known result on the unbounded domain $\mathbb{R}^2$ is linear growth, obtained by Choi and Jeong \cite{CJ} for certain perturbations of the Lamb–Chaplygin dipole. However, no mechanism for faster‑than‑linear growth had been identified in unbounded settings until very recently. Zlatoš \cite{ZA2} has now shown that on the half‑plane, using odd symmetry and boundary effects, double‑exponential growth can indeed be achieved for compactly supported initial vorticity, marking an important advance in our understanding of gradient growth on unbounded domains. More recently, Jeong, Zhou and Yao \cite{JYZ} proved superlinear gradient growth both in $\mathbb{T}^2$ and $\mathbb{R}^2$ using a steady state and a saddle point. Notably, \cite{JYZ} does not require excessive odd symmetry and holds for certain small perturbations near the steady state.

Amidst these developments, one natural geometry has remained largely unexplored: the sphere $\mathbb{S}^2$. The Euler equation on a sphere is a fundamental model for fluid flows on planetary surfaces and has attracted attention from geophysicists, meteorologists, and mathematicians over the past decades. Understanding vorticity gradient growth on curved surfaces is a step toward realistic models of planetary atmospheres and oceans. From a mathematical standpoint, the sphere is a compact manifold without boundary, yet with nontrivial curvature. Whether the double‑exponential growth mechanisms observed on the disk, the torus, or the half‑plane can operate on the sphere is a natural question. Addressing it advances our understanding of two-dimensional Euler dynamics on curved manifolds and is relevant to large-scale atmospheric and oceanic flows.

In this paper, we shall investigate this question and provide a complete answer. We will establish a double‑exponential upper bound for the vorticity gradient on the sphere and prove its sharpness by constructing explicit solutions.
  
 \subsection{Main results}
To state our main results, we first introduce some notations. For relevant discussions of the Euler equation on the sphere or spherical geometry, the reader may refer to \cite{CLW, CW24,CSMC, CGLZ, DR,GHR, LF,GE,SZ2,SZ1} and references therein.

\indent Let $|\mathbf{x}-\mathbf{y}|$ denote the Euclidean distance between $\mathbf{x}, \mathbf{y}\in\mathbb{R}^3$. The spherical coordinates $(\varphi,\theta)$ of $\mathbb{S}^2$ are
\begin{equation*}
    x_1=\cos\theta\cos\varphi,\quad x_2=\cos\theta\sin\varphi,\quad x_3=\sin\theta,\quad (\varphi,\theta)\in[-\pi,\pi)\times\left[-\frac{\pi}{2},\frac{\pi}{2}\right].
\end{equation*}
Then for any $f:\mathbb{S}^2\to \mathbb{R}$, we define $\tilde{f}(\varphi,\theta):=f(\mathbf{x})$. 
The Riemannian volume element in spherical coordinates $(\varphi,\theta)$ is
\begin{equation*}
    d\sigma(\mathbf{x})=\cos\theta \,d\varphi\mkern 1mu d\theta.
\end{equation*}
In what follows, we denote by $(\varphi',\theta')$ the spherical coordinates corresponding to $\mathbf{y}=(y_1,y_2,y_3)\in\mathbb{S}^2$. 
For $y_3\neq\pm1$, the tangent space $T_\mathbf{y}\mathbb{S}^2$ has an orthonormal basis $(\mathbf{e}_\varphi(\mathbf{y}),\mathbf{e}_\theta(\mathbf{y}))$ given by
\begin{equation*}
    \mathbf{e}_\varphi(\mathbf{y}):=\left(-\sin\varphi',\cos\varphi',0\right)=\left(-\frac{y_2}{\sqrt{1-y_3^2}},\frac{y_1}{\sqrt{1-y_3^2}},0\right),
\end{equation*}
\begin{equation*}
    \mathbf{e}_\theta(\mathbf{y}):=\left(-\sin\theta'\cos\varphi',-\sin\theta'\sin\varphi',\cos\theta'\right)=\left(-\frac{y_3y_1}{\sqrt{1-y_3^2}},-\frac{y_3y_2}{\sqrt{1-y_3^2}},\sqrt{1-y_3^2}\right).
\end{equation*}

According to \cite{DB}, denote
\begin{equation*}
    \mathcal{G}\omega(\mathbf{x},t):=\int_{\mathbb{S}^2}G(\mathbf{x},\mathbf{y})\omega(\mathbf{y},t)\,d\sigma(\mathbf{y}),\qquad G(\mathbf{x},\mathbf{y}):=\frac{1}{2\pi}\ln|\mathbf{x}-\mathbf{y}|.
\end{equation*}
The corresponding Euler equation in terms of the vorticity $\omega$ can be written as
\begin{equation}\label{Eulerw}
    \begin{cases}
        \partial_t\omega+u\cdot\nabla \omega=0,\ \ &\text{on}\ \ \mathbb{S}^2\times \mathbb{R}^+,\\
        u(\mathbf{x},t)=\nabla^\perp \mathcal{G}\omega(\mathbf{x},t),\ \ &\text{on}\ \ \mathbb{S}^2\times \mathbb{R}^+,
    \end{cases}
\end{equation}
where $\omega$ should satisfy the Gauss constraint
\begin{equation*}
    \int_{\mathbb{S}^2}\omega(\mathbf{x},t)\,d\sigma(\mathbf{x})=0,
\end{equation*}
and $u(\mathbf{x},t)=\tilde{u}_\varphi(\varphi,\theta,t)\mathbf{e}_\varphi+\tilde{u}_\theta(\varphi,\theta,t)\mathbf{e}_\theta$ must be supplemented by the following impermeability condition
\begin{equation*}
    \tilde{u}_\theta\left(\varphi,-\frac{\pi}{2},t\right)=\tilde{u}_\theta\left(\varphi,\frac{\pi}{2},t\right)=0\quad \text{for any}\quad \varphi\in[-\pi,\pi).
\end{equation*}

By    \cite[Lemma A.2]{DR}, the Biot-Savart law on the sphere is
\begin{equation*}
    u(\mathbf{x},t)=\frac{1}{2\pi}\int_{\mathbb{S}^2}\frac{\mathbf{x}\wedge\mathbf{y}}{|\mathbf{x}-\mathbf{y}|^2}\,\omega(\mathbf{y},t)\,d\sigma(\mathbf{y}).
\end{equation*}

Finally, note that from the definition of $\tilde{\omega}$, a direct computation yields
\begin{equation*}
\begin{split}
    \partial_\varphi\tilde{\omega}(\varphi,\theta)=&-(\cos\theta\sin\varphi)\partial_{x_1}\omega(\cos\theta\cos\varphi,\cos\theta\sin\varphi,\sin\theta)\\&+(\cos\theta\cos\varphi)\partial_{x_2}\omega(\cos\theta\cos\varphi,\cos\theta\sin\varphi,\sin\theta)\\
    =&-x_2\,\partial_{x_1}\omega(\mathbf{x})+x_1\,\partial_{x_2}\omega(\mathbf{x}),
\end{split}      
\end{equation*}
and
\begin{equation*}
\begin{split}
    \partial_\theta\tilde{\omega}(\varphi,\theta)=&-(\sin\theta\cos\varphi)\partial_{x_1}\omega(\cos\theta\cos\varphi,\cos\theta\sin\varphi,\sin\theta)\\&+(\sin\theta\sin\varphi)\partial_{x_2}\omega(\cos\theta\cos\varphi,\cos\theta\sin\varphi,\sin\theta)\\&+(\cos\theta)\partial_{x_3}\omega(\cos\theta\cos\varphi,\cos\theta\sin\varphi,\sin\theta)\\
    =&\frac{-x_1x_3}{\sqrt{x_1^2+x_2^2}}\,\partial_{x_1}\omega(\mathbf{x})+\frac{-x_2x_3}{\sqrt{x_1^2+x_2^2}}\,\partial_{x_2}\omega(\mathbf{x})+\sqrt{x_1^2+x_2^2}\,\partial_{x_3}\omega(\mathbf{x})
\end{split}      
\end{equation*}
for $\theta\neq\pm\pi/2$.

We now state our first main result, which provides an upper bound for the double-exponential gradient growth rate on the sphere.
\begin{theorem}\label{Main2}
    For any $\tilde{\omega}(\cdot,0)\in W^{1,\infty}\left([-\pi,\pi)\times(0,\pi/2]\right)$ or $\tilde{\omega}(\cdot,0)\in W^{1,\infty}\left([-\pi,\pi)\times[-\pi/2,\pi/2]\right)$ with $\|\tilde{\omega}(\cdot,0)\|_{L^\infty}=1$, the corresponding solution to \eqref{Eulerw} satisfies
    \begin{equation}\label{main2}
        \lim_{t\to\infty}\frac{\ln\ln\|\nabla\tilde{\omega}(\cdot,t)\|_{L^\infty{\left([-\pi,\pi)\times(0,\pi/2]\right)}}}{t}\leq\frac{2}{\pi}.
    \end{equation}
\end{theorem}

It is natural to ask whether the growth rate $2/\pi$ is optimal. By imposing odd symmetry across the equator, the dynamics on the full sphere reduce to those on the hemisphere $\mathbb{S}^2_+$.
On the hemisphere $\mathbb{S}^2_+$, the Euler equation for $\omega$ can be written as
\begin{equation}\label{Eulerw+}
    \begin{cases}
        \partial_t\omega+u\cdot\nabla \omega=0,\ \ &\text{on}\ \ \mathbb{S}^2_+\times \mathbb{R}^+,\\
        u(\mathbf{x},t)=\nabla^\perp \mathcal{G}_+\omega(\mathbf{x},t),\ \ &\text{on}\ \ \mathbb{S}^2_+\times \mathbb{R}^+.
    \end{cases}
\end{equation}
The associated Green's function and the stream function are
\begin{equation*}
    \mathcal{G}_+\omega(\mathbf{x},t):=\int_{\mathbb{S}^2}G_+(\mathbf{x},\mathbf{y})\omega(\mathbf{y},t)\,d\sigma(\mathbf{y}),\qquad G_+(\mathbf{x},\mathbf{y}):=\frac{1}{2\pi}\ln\frac{|\mathbf{x}-\mathbf{y}|}{|\mathbf{x}-\bar{\mathbf{y}}|},
\end{equation*}
where $\bar{\mathbf{y}}:=(\cos\theta'\cos\varphi',\cos\theta'\sin\varphi',-\sin\theta')=(y_1,y_2,-y_3)$.

Our second theorem shows that double-exponential gradient growth can occur for smooth solutions on the hemisphere, and demonstrates that the constant $2/\pi$ in \eqref{main2} is indeed attainable.
\begin{theorem}\label{Main1}
    There exists $\tilde{\omega}(\cdot,0)\in C^\infty\left([-\pi,\pi)\times\left(0,\pi/2\right]\right)$ with $\|\tilde{\omega}(\cdot,0)\|_{L^\infty}=1$ such that the corresponding solution to \eqref{Eulerw+} satisfies
    \begin{equation}\label{main1}
        \lim_{t\to\infty}\frac{\ln\ln\|\nabla\tilde{\omega}(\cdot,t)\|_{L^\infty\left([-\pi,\pi)\times(0,\pi/2]\right)}}{t}=\frac{2}{\pi}.
    \end{equation}
\end{theorem}

\begin{remark}
    A simple scaling argument shows that the result is not limited to the normalization $\|\tilde{\omega}(\cdot,0)\|_{L^\infty}=1$. If the initial vorticity satisfies $\|\tilde{\omega}(\cdot,0)\|_{L^\infty}=b$ for some $b>0$, then the conclusion of Theorem \ref{Main1} can be adapted accordingly. More precisely, the double-exponential growth rate $2/\pi$ in \eqref{main1} is replaced by $2b/\pi$.
\end{remark}
\begin{remark}
    Similar to the half-space result of Zlatoš \cite{ZA2}, the gradient of the solution constructed in our Theorem~\ref{Main1} has a singularity at the boundary of the hemisphere (i.e., the equator). Whether one can construct smooth solutions on the whole plane, the torus, and the sphere, which achieve double‑exponential gradient growth remains an outstanding open problem.
\end{remark}

\begin{figure}[htbp]
    \centering
    \includegraphics[width=0.5\textwidth]{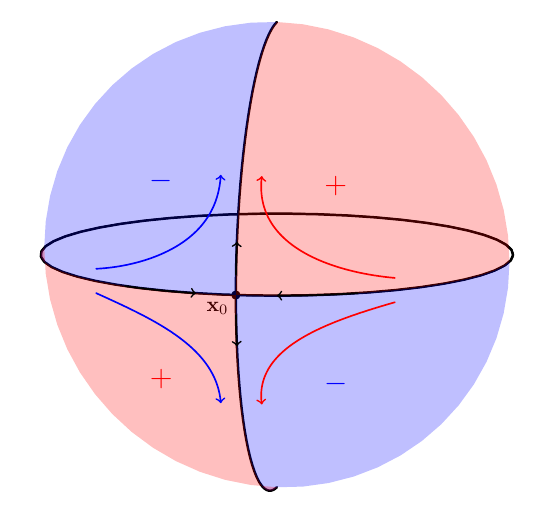}
    \caption{Odd-odd symmetric vorticity drives particles toward the symmetry axis $\{\varphi=0\}$.}
    \label{fig:1}
\end{figure}

     The proof of our main results is built on the strategy developed in \cite{KV,ZA1,ZA2}, but requires substantial adaptations to the spherical geometry. The proof of our upper bound (Theorem~\ref{Main2}) follows the general strategy of \cite{ZA2} (which relies on precise estimates on the corresponding Dirichlet Green’s functions  for uniformly smooth domains in $\mathbb{R}^2$) and does not require any symmetry assumption. For the construction achieving the sharp rate (Theorem~\ref{Main1}), we impose odd-odd symmetry on the vorticity with respect to both the equator and a meridian, which reduces the dynamics to the quarter-sphere $\mathbb{S}^2_q$ and creates an effective corner where vorticity gradients can be compressed. Under this symmetry, the Biot–Savart law on the sphere takes a form that allows us to derive sharp asymptotic estimates for the velocity components $\tilde u_\varphi$ and $\tilde u_\theta$ in spherical coordinates. A careful decomposition of the integral kernels reveals that the leading-order contributions come from a region where both angular variables are large compared to the observation point, while the remainders are controlled either by constants or by the gradient of the vorticity. This analysis leads to precise expressions for the velocity field near the corner, which are then used to track the evolution of a carefully constructed family of shrinking domains $\alpha(t)\Omega_{\varepsilon(t)}$. By choosing the scaling function $\alpha(t)$ and the shrinking parameter $\varepsilon(t)$ appropriately, we ensure that the boundary of the set where the vorticity equals $1$ never enters these domains. The resulting dynamics forces $\varepsilon(t)$ to satisfy a differential inequality whose solution exhibits double‑exponential decay with the optimal rate $2/\pi$, thereby establishing the sharpness of the upper bound.

      The rest of this paper is organized as follows: In Section 2, we establish a key lemma concerning the velocity and prove Theorem \ref{Main2}. In Section 3, we proceed to the proof of Theorem \ref{Main1}. Finally, in Section 4, we briefly mention the corresponding results on the rotating sphere.

\section{Proof of Theorem \ref{Main2}}
We start this section by giving some necessary notations and properties, and fix some assumptions regarding the sphere.

\subsection{Properties and Notations on the Sphere}
For positive functions $f$ and $g$, we write $f\lesssim g$ if there exists an absolute constant $C>0$ such that $f\leq Cg$. The notation $f\sim g$ means $f\lesssim g$ and $g\lesssim f$. For $k\in \mathbb{R}$, we write $k=O(g)$ to denote $|k|\lesssim g$.

By \cite[Proposition 1]{CAG}, for any sufficiently regular solution $\omega(\cdot,t)$ of \eqref{Eulerw}, the following conservation law holds:
\begin{equation*}
    \omega(\cdot,t)\in\mathcal{R}_{\omega(\cdot,0)}\quad\text{for any}\quad t\in[0,\infty),
\end{equation*}
where for any measurable (with respect to the Riemannian measure) function $v:\mathbb{S}^2\to\mathbb{R}$, $\mathcal{R}_v$ denotes the rearrangement class of $v$ on $\mathbb{S}^2$, i.e.,
\begin{equation*}
    \mathcal{R}_v:=\{g:|\{\mathbf{x}\in\mathbb{S}^2:g(\mathbf{x})>s\}|=|\{\mathbf{x}\in\mathbb{S}^2:v(\mathbf{x})>s\}|,\;\forall s\in\mathbb{R}\}.
\end{equation*}
Hence, for any $p\in[1,\infty]$, $\|\omega(\cdot,t)\|_{L^p(\mathbb{S}^2)}$ is conserved.

We denote
\begin{equation*}
    \tilde{\mathbf{x}}:=(\cos\theta\cos\varphi,-\cos\theta\sin\varphi,\sin\theta)=(x_1,-x_2,x_3),\quad
    \bar{\mathbf{x}}:=(\cos\theta\cos\varphi,\cos\theta\sin\varphi,-\sin\theta)=(x_1,x_2,-x_3).
\end{equation*}
Let $\mathbb{S}^2_{q}:=\{(x_1,x_2,x_3)\in\mathbb{S}^2: x_2>0,\;x_3>0\}$ be the quarter-sphere. Consider an odd-odd symmetric vorticity $\omega(\cdot)\in L^\infty(\mathbb{S}^2)$, namely,
\begin{equation*}
    \omega(\mathbf{x})=-\omega(\tilde{\mathbf{x}})=-\omega(\bar{\mathbf{x}}),
\end{equation*}
or equivalently $\tilde{\omega}(\varphi,\theta)=-\tilde{\omega}(-\varphi,\theta)=-\tilde{\omega}(\varphi,-\theta)$. Then for $\theta\neq\pm\pi/2$, we have
\begin{equation}\label{uvarphi}
\begin{split}
    \tilde{u}_\varphi(\varphi,\theta) &= u(\mathbf{x})\cdot \mathbf{e}_\varphi \\
    &= \frac{1}{2\pi}\int_{\mathbb{S}^2}\frac{-x_1^2y_3-x_2^2y_3+x_2x_3y_2+x_1x_3y_1}{|\mathbf{x}-\mathbf{y}|^2\sqrt{1-x_3^2}}\,\omega(\mathbf{y})\,d\sigma(\mathbf{y}) \\
    &= \frac{1}{2\pi}\int_{\mathbb{S}^2\cap\{y_2>0\}}\left(\frac{-x_1^2y_3-x_2^2y_3+x_2x_3y_2+x_1x_3y_1}{|\mathbf{x}-\mathbf{y}|^2\sqrt{1-x_3^2}} - \frac{-x_1^2y_3-x_2^2y_3-x_2x_3y_2+x_1x_3y_1}{|\mathbf{x}-\tilde{\mathbf{y}}|^2\sqrt{1-x_3^2}}\right)\omega(\mathbf{y})\,d\sigma(\mathbf{y}) \\
    &= \frac{1}{2\pi}\int_{\mathbb{S}^2_{q}}\frac{2x_2\Bigl(x_3y_2\bigl(x_2^2+y_2^2+(x_1-y_1)^2+(x_3-y_3)^2\bigr)+2y_2\bigl(-x_1^2y_3-x_2^2y_3+x_1x_3y_1\bigr)\Bigr)}{|\mathbf{x}-\mathbf{y}|^2|\mathbf{x}-\tilde{\mathbf{y}}|^2\sqrt{1-x_3^2}}\,\omega(\mathbf{y})\,d\sigma(\mathbf{y}) \\
    &\quad -\frac{1}{2\pi}\int_{\mathbb{S}^2_{q}}\frac{2x_2\Bigl(x_3y_2\bigl(x_2^2+y_2^2+(x_1-y_1)^2+(x_3+y_3)^2\bigr)+2y_2\bigl(x_1^2y_3+x_2^2y_3+x_1x_3y_1\bigr)\Bigr)}{|\mathbf{x}-\bar{\mathbf{y}}|^2|\mathbf{x}-\bar{\tilde{\mathbf{y}}}|^2\sqrt{1-x_3^2}}\,\omega(\mathbf{y})\,d\sigma(\mathbf{y}) \\
    &= \frac{x_2}{\pi\sqrt{1-x_3^2}}\int_{\mathbb{S}^2_{q}}\left(\frac{2y_2(x_3-y_3)}{|\mathbf{x}-\mathbf{y}|^2|\mathbf{x}-\tilde{\mathbf{y}}|^2}-\frac{2y_2(x_3+y_3)}{|\mathbf{x}-\bar{\mathbf{y}}|^2|\mathbf{x}-\bar{\tilde{\mathbf{y}}}|^2}\right)\omega(\mathbf{y})\,d\sigma(\mathbf{y}),
\end{split}
\end{equation}
and
\begin{equation}\label{utheta}
\begin{split}
    \tilde{u}_\theta(\varphi,\theta) &= u(\mathbf{x})\cdot \mathbf{e}_\theta \\
    &= \frac{1}{2\pi}\int_{\mathbb{S}^2}\frac{x_1y_2-x_2y_1}{|\mathbf{x}-\mathbf{y}|^2\sqrt{1-x_3^2}}\,\omega(\mathbf{y})\,d\sigma(\mathbf{y}) \\
    &= \frac{1}{2\pi}\int_{\mathbb{S}^2_{q}}\left(\frac{4x_3y_3(x_1y_2-x_2y_1)}{|\mathbf{x}-\mathbf{y}|^2|\mathbf{x}-\bar{\mathbf{y}}|^2\sqrt{1-x_3^2}} - \frac{4x_3y_3(-x_1y_2-x_2y_1)}{|\mathbf{x}-\tilde{\mathbf{y}}|^2|\mathbf{x}-\tilde{\bar{\mathbf{y}}}|^2\sqrt{1-x_3^2}}\right)\omega(\mathbf{y})\,d\sigma(\mathbf{y}) \\
    &= \frac{2x_3}{\pi\sqrt{1-x_3^2}}\int_{\mathbb{S}^2_{q}}\left(\frac{y_3\bigl((x_1-y_1)y_2-(x_2-y_2)y_1\bigr)}{|\mathbf{x}-\mathbf{y}|^2|\mathbf{x}-\bar{\mathbf{y}}|^2}+\frac{y_3\bigl((x_2+y_2)x_1-(x_1-y_1)x_2\bigr)}{|\mathbf{x}-\tilde{\mathbf{y}}|^2|\mathbf{x}-\tilde{\bar{\mathbf{y}}}|^2}\right)\omega(\mathbf{y})\,d\sigma(\mathbf{y}).
\end{split}
\end{equation}

Under the odd-odd symmetry assumption, from the above expressions we obtain that $\tilde{u}_\varphi$ is odd in $\varphi$ and even in $\theta$, while $\tilde{u}_\theta$ is even in $\varphi$ and odd in $\theta$. In particular,
\begin{equation*}
    \tilde{u}_\varphi(0,\theta)=0\quad\text{for any}\quad \theta\in\left(-\frac{\pi}{2},\frac{\pi}{2}\right),
\end{equation*}
and
\begin{equation*}
    \tilde{u}_\theta(\varphi,0)=0\quad\text{for any}\quad \varphi\in[-\pi,\pi).
\end{equation*}

Let $\mathbf{e}_2=(0,1,0)$. When $\theta=\pm\pi/2$, the component of $u(\mathbf{x})$ in the $x_2$-axis direction is
\begin{equation*}
    u_2(\mathbf{x})=u(\mathbf{x})\cdot\mathbf{e}_2=\int_{\mathbb{S}^2}\frac{y_1}{2\pi|\mathbf{x}-\mathbf{y}|^2}\,\omega(\mathbf{y})\,d\sigma(\mathbf{y})=0.
\end{equation*}
Therefore, the odd-odd symmetry is preserved for all time.

\subsection{Velocity estimates}

In this subsection we derive asymptotic expressions for the velocity components $\tilde u_\varphi$ and $\tilde u_\theta$ under the odd-odd symmetry assumption for vorticity. These estimates are crucial for the subsequent analysis of particle trajectories and for quantifying the growth of the vorticity gradient. Following the strategy of \cite{KV,ZA1, ZA2}, we split the Biot–Savart integral into a dominant part coming from a region where both $\varphi'$ and $\theta'$ are large compared to $\varphi$ and $\theta$, and several remainder regions that contribute only lower-order terms. The leading contribution turns out to be expressed in terms of the integral of $\frac{y_2y_3}{|\mathbf{y}-\mathbf{x}_0|^4}$ over a quarter-sphere region $Q(3\varphi/2,3\theta/2)$. The remainders are controlled either by a constant or by quantities involving the gradient of $\omega$, leading to the bounds stated in Lemma \ref{Velocity estimate}. Compared with the analysis in \cite{KV,ZA1, ZA2, ZA2}, the spherical setting introduces additional geometric factors due to curvature and the chordal distance, but the overall structure of the estimates remains similar.  

Denote $\mathbf{x}_0=(1,0,0)$, and for $k_i\geq0$, $i=1,2$, define
\begin{equation*}
    Q(k_1,k_2):=\left\{(\cos\theta'\cos\varphi',\cos\theta'\sin\varphi',\sin\theta'):\varphi'\in(k_1,\pi),\;\theta'\in\left(k_2,\frac{\pi}{2}\right)\right\}.
\end{equation*}
Then the following key lemma holds.
\begin{lemma}\label{Velocity estimate}
    There exists a universal constant $C_1\geq1$ such that if $\tilde{\omega}\in L^\infty([-\pi,\pi)\times[-\pi/2,\pi/2])$ is odd in $\varphi$ and $\theta$, then for each $(\varphi,\theta)\in[0,1]^2\setminus\{(0,0)\}$, we have
    \begin{equation}\label{velocity estimate1}
        \tilde{u}_\varphi(\varphi,\theta)=\frac{x_2}{\sqrt{1-x_3^2}}\left(-\frac{4}{\pi}\int_{Q\left(\frac{3\varphi}{2},\frac{3\theta}{2}\right)}\frac{y_2y_3}{|\mathbf{y}-\mathbf{x}_0|^4}\omega(\mathbf{y})d\sigma(\mathbf{y})+B_\varphi(\varphi,\theta)\right),
    \end{equation}
    \begin{equation}\label{velocity estimate2}
        \tilde{u}_\theta(\varphi,\theta)=\frac{x_3}{\sqrt{1-x_3^2}}\left(\frac{4x_1}{\pi}\int_{Q\left(\frac{3\varphi}{2},\frac{3\theta}{2}\right)}\frac{y_2y_3}{|\mathbf{y}-\mathbf{x}_0|^4}\omega(\mathbf{y})d\sigma(\mathbf{y})+B_\theta(\varphi,\theta)\right),
    \end{equation}
    where $B_\varphi$ and $B_\theta$ depend on $\tilde{\omega}$ and  satisfy
    \begin{equation*}
        \begin{split}
            &|B_\varphi(\varphi,\theta)|\leq C_1\|\tilde{\omega}\|_{L^\infty}\left(1+\min\left\{\ln\left(1+\frac{\theta}{\varphi}\right),\frac{\|\nabla\tilde{\omega}\|_{L^\infty(\left[0,{3\theta}/{2}\right]^2)}}{\|\tilde{\omega}\|_{L^\infty}}\theta\right\}\right),\\
            &|B_\theta(\varphi,\theta)|\leq C_1\|\tilde{\omega}\|_{L^\infty}\left(1+\min\left\{\ln\left(1+\frac{\varphi}{\theta}\right),\frac{\|\nabla\tilde{\omega}\|_{L^\infty(\left[0,{3\varphi}/{2}\right]^2)}}{\|\tilde{\omega}\|_{L^\infty}}\varphi\right\}\right).
        \end{split}
    \end{equation*}
\end{lemma}

\begin{proof}
    For any $\mathbf{x}\in\{(\cos\theta\cos\varphi,\cos\theta\sin\varphi,\sin\theta):(\varphi,\theta)\in[-1,1]^2\}$ and $\mathbf{y}\in\mathbb{S}^2$, the following relations hold:
    \begin{equation}\label{distance}
        |\mathbf{x}-\mathbf{y}|\sim\sqrt{(\varphi-\varphi')^2+(\theta-\theta')^2},\quad |x_1-y_1|+|x_2-y_2|\sim|\varphi-\varphi'|,\quad |x_3-y_3|\sim|\theta-\theta'|.
    \end{equation}
    We first consider $\tilde{u}_\varphi$ and begin by partitioning $\mathbb{S}^2_{q}$ into three disjoint regions:
    \[
    Q\left(\frac{3\varphi}{2},\frac{3\theta}{2}\right),\qquad Q_2,\qquad Q_3,
    \]
    where
    \begin{equation*}
        \begin{split}
            &Q_2=\left\{(\cos\theta'\cos\varphi',\cos\theta'\sin\varphi',\sin\theta'):\varphi'\in\left(0,\frac{3\varphi}{2}\right],\,\theta'\in\left(0,\frac{\pi}{2}\right)\right\},\\
            &Q_3=\left\{(\cos\theta'\cos\varphi',\cos\theta'\sin\varphi',\sin\theta'):\varphi'\in\left[\frac{3\varphi}{2},\pi\right),\,\theta'\in\left(0,\frac{3\theta}{2}\right]\right\}.
        \end{split}
    \end{equation*}

    \textbf{Estimates on $Q(3\varphi/2,3\theta/2)$.} 
    In this region we have
    \begin{equation*}
        |\mathbf{y}-\mathbf{x}_0|\lesssim|\mathbf{x}-\mathbf{y}|,\quad |\mathbf{y}-\mathbf{x}_0|\lesssim|\mathbf{x}-\tilde{\mathbf{y}}|.
    \end{equation*}
    Together with $x_3\leq|\mathbf{x}-\mathbf{x}_0|$, $y_2\leq|\mathbf{y}-\mathbf{x}_0|$ and \eqref{distance}, these inequalities yield
    \begin{equation*}
    \begin{split}
         \int_{Q\left(\frac{3\varphi}{2},\frac{3\theta}{2}\right)}\frac{2x_3y_2}{|\mathbf{x}-\mathbf{y}|^2|\mathbf{x}-\tilde{\mathbf{y}}|^2}\omega(\mathbf{y})d\sigma(\mathbf{y})&\lesssim\|\tilde{\omega}\|_{L^\infty}\int_{Q\left(\frac{3\varphi}{2},\frac{3\theta}{2}\right)}\frac{|\mathbf{x}-\mathbf{x}_0|}{|\mathbf{y}-\mathbf{x}_0|^3}d\sigma(\mathbf{y})\\
         &\lesssim\|\tilde{\omega}\|_{L^\infty}\int^{\frac{\pi}{2}}_{\frac{3\theta}{2}}\int^{\pi}_{\frac{3\varphi}{2}}\frac{(\varphi^2+\theta^2)^{\frac{1}{2}}}{(\varphi'^{\mkern 2mu 2}+\theta'^{\mkern 2mu 2})^{\frac{3}{2}}}d\varphi'd\theta'\\
         &\lesssim\|\tilde{\omega}\|_{L^\infty}.
    \end{split}       
    \end{equation*}
    On the other hand, a direct expansion gives
    \begin{equation*}
    \begin{split}
        \frac{1}{|\mathbf{x}-\mathbf{y}|^2|\mathbf{x}-\tilde{\mathbf{y}}|^2}-\frac{1}{|\mathbf{y}-\mathbf{x}_0|^4}=&O\left(\frac{|\mathbf{y}-\mathbf{x}_0|^4-|(\mathbf{x}-\mathbf{x}_0)-(\mathbf{y}-\mathbf{x}_0)|^2|(\mathbf{x}-\mathbf{x}_0)-(\tilde{\mathbf{y}}-\mathbf{x}_0)|^2}{|\mathbf{y}-\mathbf{x}_0|^8}\right)\\
        =&O\left(\frac{|\mathbf{x}-\mathbf{x}_0||\mathbf{y}-\mathbf{x}_0|^3}{|\mathbf{y}-\mathbf{x}_0|^8}\right)\\
        =&O\left(\frac{|\mathbf{x}-\mathbf{x}_0|}{|\mathbf{y}-\mathbf{x}_0|^5}\right),
    \end{split}        
    \end{equation*}
    which implies that
    \begin{equation*}
    \begin{split}
        \int_{Q\left(\frac{3\varphi}{2},\frac{3\theta}{2}\right)}\left(\frac{1}{|\mathbf{x}-\mathbf{y}|^2|\mathbf{x}-\tilde{\mathbf{y}}|^2}-\frac{1}{|\mathbf{y}-\mathbf{x}_0|^4}\right)y_2y_3\omega(\mathbf{y})d\sigma(\mathbf{y})
        \lesssim&\int_{Q\left(\frac{3\varphi}{2},\frac{3\theta}{2}\right)}\frac{|\mathbf{x}-\mathbf{x}_0||\mathbf{y}-\mathbf{x}_0|^2}{|\mathbf{y}-\mathbf{x}_0|^5}\omega(\mathbf{y})d\sigma(\mathbf{y})\\
        \leq&\|\tilde{\omega}\|_{L^\infty}\int_{Q\left(\frac{3\varphi}{2},\frac{3\theta}{2}\right)}\frac{|\mathbf{x}-\mathbf{x}_0|}{|\mathbf{y}-\mathbf{x}_0|^3}d\sigma(\mathbf{y})\\
        \lesssim&\|\tilde{\omega}\|_{L^\infty}.
    \end{split}         
    \end{equation*}
    Combining the above estimates, we obtain
    \begin{equation}\label{Q1f1}
        \int_{Q\left(\frac{3\varphi}{2},\frac{3\theta}{2}\right)}\frac{2y_2(x_3-y_3)}{|\mathbf{x}-{\mathbf{y}}|^2|\mathbf{x}-{\tilde{\mathbf{y}}}|^2}\omega(\mathbf{y})d\sigma(\mathbf{y})=-2\int_{Q\left(\frac{3\varphi}{2},\frac{3\theta}{2}\right)}\frac{y_2y_3}{|\mathbf{y}-\mathbf{x}_0|^4}\omega(\mathbf{y})d\sigma(\mathbf{y})+O(\|\tilde{\omega}\|_{L^\infty}).
    \end{equation}
    Similarly, for the remaining part of the integral over $Q\left({3\varphi}/{2},{3\theta}/{2}\right)$ in $\tilde{u}_\varphi(\varphi,\theta)$, we use the bounds
    \begin{equation*}
        |\mathbf{y}-\mathbf{x}_0|\lesssim|\mathbf{x}-\bar{\mathbf{y}}|,\quad |\mathbf{y}-\mathbf{x}_0|\lesssim|\mathbf{x}-\bar{\tilde{\mathbf{y}}}|
    \end{equation*}
    and the expansion
    \begin{equation*}
        \frac{1}{|\mathbf{x}-\bar{\mathbf{y}}|^2|\mathbf{x}-\bar{\tilde{\mathbf{y}}}|^2}-\frac{1}{|\mathbf{y}-\mathbf{x}_0|^4}=O\left(\frac{|\mathbf{x}-\mathbf{x}_0|}{|\mathbf{y}-\mathbf{x}_0|^5}\right).
    \end{equation*}
    Consequently,
    \begin{equation}\label{Q1f2}
        -\int_{Q\left(\frac{3\varphi}{2},\frac{3\theta}{2}\right)}\frac{2y_2(x_3+y_3)}{|\mathbf{x}-\bar{{\mathbf{y}}}|^2|\mathbf{x}-{\bar{\tilde{\mathbf{y}}}}|^2}\omega(\mathbf{y})d\sigma(\mathbf{y})=-2\int_{Q\left(\frac{3\varphi}{2},\frac{3\theta}{2}\right)}\frac{y_2y_3}{|\mathbf{y}-\mathbf{x}_0|^4}\omega(\mathbf{y})d\sigma(\mathbf{y})+O(\|\tilde{\omega}\|_{L^\infty}).
    \end{equation}

    \textbf{Estimates on $Q_2$ and $Q_3$.} 
    We now turn to the integral over $Q_2$. Using \eqref{distance} and performing the change of variables $(\varphi'',\theta'')=(\varphi'-\varphi,\theta'-\theta)$, we obtain
    \begin{equation}\label{Q2f1}
    \begin{split}
         \int_{Q_2}\frac{2y_2(x_3-y_3)}{|\mathbf{x}-{\mathbf{y}}|^2|\mathbf{x}-{\tilde{\mathbf{y}}}|^2}\omega(\mathbf{y})d\sigma(\mathbf{y})&\lesssim\|\tilde{\omega}\|_{L^\infty}\int^\frac{\pi}{2}_0\int^\frac{3\varphi}{2}_0\frac{x_2|\theta'-\theta|\cos\theta'}{\big((\varphi'-\varphi)^2+(\theta'-\theta)^2\big)\big(x_2^2+(\theta'-\theta)^2\big)}d\varphi'd\theta'\\
         &\lesssim\|\tilde{\omega}\|_{L^\infty}\int^\frac{\pi}{2}_0\int^\varphi_0\frac{\cos\theta\sin\varphi\,\theta''}{\big(\varphi''^{\mkern 1mu 2}+\theta''^{\mkern 1mu 2}\big)\big((\cos\theta\sin\varphi)^2+\theta''^{\mkern 1mu 2}\big)}d\varphi''d\theta''\\
         &\lesssim\|\tilde{\omega}\|_{L^\infty}\int^\frac{\pi}{2}_0\int^\varphi_0\frac{\varphi\,\theta''}{\big(\varphi''^{\mkern 1mu 2}+\theta''^{\mkern 1mu 2}\big)\big(\varphi^2+\theta''^{\mkern 1mu 2}\big)}d\varphi''d\theta''\\
         &\lesssim\|\tilde{\omega}\|_{L^\infty}\int^\frac{\pi}{2}_0\frac{\varphi}{\varphi^2+\theta''^{\mkern 1mu 2}}\arctan\frac{\varphi}{\theta''}d\theta''\\
         &\lesssim\|\tilde{\omega}\|_{L^\infty}\arctan\frac{1}{\varphi}\\
         &\lesssim\|\tilde{\omega}\|_{L^\infty}.
    \end{split}       
    \end{equation}
    Analogously, with the change $(\varphi'',\theta'')=(\varphi'-\varphi,\theta'+\theta)$, we get
    \begin{equation}\label{Q2f2}
    \begin{split}
        \int_{Q_2}\frac{2y_2(x_3+y_3)}{|\mathbf{x}-\bar{{\mathbf{y}}}|^2|\mathbf{x}-{\bar{\tilde{\mathbf{y}}}}|^2}\omega(\mathbf{y})d\sigma(\mathbf{y})&\lesssim\|\tilde{\omega}\|_{L^\infty}\int^{\pi}_0\int^\varphi_0\frac{\varphi\,\theta''}{\big(\varphi''^{\mkern 1mu 2}+\theta''^{\mkern 1mu 2}\big)\big(\varphi^2+\theta''^{\mkern 1mu 2}\big)}d\varphi''d\theta''\\
        &\lesssim\|\tilde{\omega}\|_{L^\infty}.
    \end{split}        
    \end{equation}
    Next, we consider the integral over $Q_3$. Applying \eqref{distance} and the change $(\varphi'',\theta'')=(\varphi'-\varphi,\theta'-\theta)$ yields
    \begin{equation}\label{Q3f1}
        \begin{split}
            \int_{Q_3}\frac{2y_2(x_3-y_3)}{|\mathbf{x}-{\mathbf{y}}|^2|\mathbf{x}-{\tilde{\mathbf{y}}}|^2}\omega(\mathbf{y})d\sigma(\mathbf{y})&\lesssim\|\tilde{\omega}\|_{L^\infty}\int^\frac{3\theta}{2}_0\int^\pi_\frac{3\varphi}{2}\frac{\varphi'\,|\theta'-\theta|\cos\theta'}{\big((\varphi'-\varphi)^2+(\theta'-\theta)^2\big)^2}d\varphi'd\theta'\\            &\lesssim\|\tilde{\omega}\|_{L^\infty}\int^\theta_0\int^\pi_\frac{\varphi}{2}\frac{\varphi''\,\theta''}{\varphi''^{\mkern 1mu 2}+\theta''^{\mkern 1mu 2}}d\varphi''d\theta''\\
            &\lesssim\|\tilde{\omega}\|_{L^\infty}\int^\theta_0\frac{\theta''}{\varphi^2+\theta''^{\mkern 1mu 2}}d\theta''\\
            &\lesssim\|\tilde{\omega}\|_{L^\infty}\left(\ln\left(1+\frac{\theta}{\varphi}\right)+1\right).
        \end{split}
    \end{equation}
    When $\varphi\geq\theta$, we have $\ln(1+\theta/\varphi)\leq\ln2$ in the above expression, so the bound is of order $\|\tilde{\omega}\|_{L^\infty}$. If instead $\varphi<\theta$, we further partition $Q_3$ into the following two subregions:
    \begin{equation*}
        \begin{split}
            &Q_3'=\left\{(\cos\theta'\cos\varphi',\cos\theta'\sin\varphi',\sin\theta'):\varphi'\in\left[\frac{3\theta}{2},\pi\right),\,\theta'\in\left(0,\frac{3\theta}{2}\right]\right\},\\
            &Q_4=\left\{(\cos\theta'\cos\varphi',\cos\theta'\sin\varphi',\sin\theta'):\varphi'\in\left[\frac{3\varphi}{2},\frac{3\theta}{2}\right],\,\theta'\in\left(0,\frac{3\theta}{2}\right]\right\}.
        \end{split}
    \end{equation*}
    Following the same computation as in \eqref{Q3f1}, we obtain
    \begin{equation}\label{Q3'f1}
        \int_{Q_3'}\frac{2y_2(x_3-y_3)}{|\mathbf{x}-{\mathbf{y}}|^2|\mathbf{x}-{\tilde{\mathbf{y}}}|^2}\omega(\mathbf{y})d\sigma(\mathbf{y})\lesssim\|\tilde{\omega}\|_{L^\infty}(1+\ln2).
    \end{equation}
    For the region $Q_4$, set $M:=\|\nabla\tilde{\omega}\|_{L^\infty([0,3\theta/2]^2)}$ and decompose $\omega$ as
    \begin{equation*}
        \omega(\mathbf{y})=\omega_1(y_1,y_2)+\omega_2(\mathbf{y}),
    \end{equation*}
    where $\omega_1(y_1,y_2)=\omega\left(\frac{y_1\sqrt{1-x_3^2}}{\sqrt{y_1^2+y_2^2}},\frac{y_2\sqrt{1-x_3^2}}{\sqrt{y_1^2+y_2^2}},x_3\right)$ (hence $\tilde{\omega}_1(\varphi')=\tilde{\omega}(\varphi',\theta)$). Then we have
    \begin{equation*}
        |\omega_2(\mathbf{y})|=|\omega(\mathbf{y})-\omega_1(y_1,y_2)|\leq M|\theta-\theta'|.
    \end{equation*}
    By the odd symmetry, it follows that
    \begin{equation*}
        |\omega_1(y_1,y_2)|=\left|\omega_1(y_1,y_2)-\omega_1\left(\sqrt{1-x_3^2},0\right)\right|\leq M|\varphi'|.
    \end{equation*}
    Therefore, using the change of variable $\theta''=\theta'-\theta$ together with $y_2\leq|\mathbf{x}-\tilde{\mathbf{y}}|$ and $|x_3-y_3|\leq|\mathbf{x}-\mathbf{y}|$, we obtain
    \begin{equation}\label{Q4w1}
        \begin{split}
            \int_{Q_4}\frac{y_2(x_3-y_3)}{|\mathbf{x}-\mathbf{y}|^2|\mathbf{x}-\tilde{\mathbf{y}}|^2}\omega_1(y_1,y_2)d\sigma(\mathbf{y})\lesssim&\int^\frac{3\theta}{2}_0\int^\frac{3\theta}{2}_\frac{3\varphi}{2}\frac{M\varphi'}{\varphi'^{\mkern 2mu 2}+(\theta'-\theta)^2}d\varphi'd\theta'\\
            \leq&\int^\theta_0\int^\frac{3\theta}{2}_\frac{3\varphi}{2}\frac{M\varphi'}{\varphi'^{\mkern 2mu 2}+\theta''^{\mkern 1mu 2}}d\varphi'd\theta''\\
            =&\frac{M}{2}\left(\theta\ln\left(\theta^2+\frac{9\theta^2}{4}\right)-2\theta+3\theta\arctan\frac{2\theta}{3\theta}\right)\\
            &-\frac{M}{2}\left(\theta\ln\left(\theta^2+\frac{9\varphi^2}{4}\right)-2\theta+3\varphi\arctan\frac{2\theta}{3\varphi}\right)\\
            \lesssim&M\theta,
        \end{split}
    \end{equation}
    and
    \begin{equation}\label{Q4w2}
    \begin{split}
        \int_{Q_4}\frac{y_2(x_3-y_3)}{|\mathbf{x}-\mathbf{y}|^2|\mathbf{x}-\tilde{\mathbf{y}}|^2}\omega_2(\mathbf{y})d\sigma(\mathbf{y})\lesssim&M\int^\theta_0\int^\frac{3\theta}{2}_\frac{3\varphi}{2}\frac{\varphi'\theta''^{\mkern 1mu 2}}{\big(\varphi'^{\mkern 2mu 2}+\theta''^{\mkern 1mu 2}\big)^2}d\varphi'd\theta''\\
        \lesssim&M\int^\theta_0\frac{\theta''^{\mkern 1mu 2}}{\varphi^2+\theta''^{\mkern 1mu 2}}d\theta''\\
        \lesssim&M\theta.
    \end{split}        
    \end{equation}
    For the second term in the integral on the right-hand side of \eqref{uvarphi}, when $\varphi<\theta$ we do not need to split $Q_3$ into $Q_3'$ and $Q_4$. Instead, we let $(\varphi'',\theta'')=(\varphi'-\varphi,\theta'+\theta)$ and apply \eqref{distance} again to obtain
    \begin{equation}\label{Q3f2}
        \begin{split}
            \int_{Q_3}\frac{2y_2(x_3+y_3)}{|\mathbf{x}-\bar{{\mathbf{y}}}|^2|\mathbf{x}-{\bar{\tilde{\mathbf{y}}}}|^2}\omega(\mathbf{y})d\sigma(\mathbf{y})&\lesssim\|\tilde{\omega}\|_{L^\infty}\int^\frac{3\theta}{2}_0\int^\pi_\frac{3\varphi}{2}\frac{(\varphi'-\varphi)(\theta'+\theta)\cos\theta'}{\big((\varphi'-\varphi)^2+(\theta'-\theta)^2\big)^2}d\varphi'd\theta'\\
            &\leq\|\tilde{\omega}\|_{L^\infty}\int^\frac{5\theta}{2}_\theta\int^\pi_\frac{\varphi}{2}\frac{\varphi''\theta''}{\big(\varphi''^{\mkern 1mu 2}+\theta''^{\mkern 1mu 2}\big)^2}d\varphi''d\theta''\\
            &\leq\|\tilde{\omega}\|_{L^\infty}\int^\frac{5\theta}{2}_\theta\frac{4\theta''}{\varphi^2+4\theta''^2}d\theta''\\
            &=\frac{1}{2}\|\tilde{\omega}\|_{L^\infty}\ln\frac{\varphi^2+25\theta^2}{\varphi^2+4\theta^2}\\
            &\lesssim\|\tilde{\omega}\|_{L^\infty}.
        \end{split}
    \end{equation}
    Combining \eqref{Q1f1}–\eqref{Q3f2}, we finally arrive at \eqref{velocity estimate1}.

    \textbf{Estimates for $\tilde{u}_\theta$.} 
    The analysis for $\tilde{u}_\theta(\varphi,\theta)$ proceeds similarly. We begin by partitioning $\mathbb{S}^2$ into $Q(3\varphi/2,3\theta/2)$, $\tilde{Q}_2$ and $\tilde{Q}_3$, where
    \begin{equation*}
        \begin{split}
            &\tilde{Q}_2=\left\{(\cos\theta'\cos\varphi',\cos\theta'\sin\varphi',\sin\theta'):\varphi'\in\left(0,\pi\right),\,\theta'\in\left(0,\frac{3\theta}{2}\right]\right\},\\
            &\tilde{Q}_3=\left\{(\cos\theta'\cos\varphi',\cos\theta'\sin\varphi',\sin\theta'):\varphi'\in\left(0,\frac{3\varphi}{2}\right],\,\theta'\in\left[\frac{3\theta}{2},\frac{\pi}{2}\right)\right\}.
        \end{split}
    \end{equation*}
    On $Q(3\varphi/2,3\theta/2)$, analogous to the previous calculation for $\tilde{u}_\varphi$, we have
    \begin{equation}\label{Q1f1'}
    \begin{split}
        \int_{Q\left(\frac{3\varphi}{2},\frac{3\theta}{2}\right)}\frac{y_3\big((x_1-y_1)y_2-(x_2-y_2)y_1\big)}{|\mathbf{x}-{\mathbf{y}}|^2|\mathbf{x}-{\bar{\mathbf{y}}}|^2}\omega(\mathbf{y})d\sigma(\mathbf{y})&=\int_{Q\left(\frac{3\varphi}{2},\frac{3\theta}{2}\right)}\frac{x_1y_2y_3}{|\mathbf{x}-\mathbf{y}|^2|\mathbf{x}-\bar{\mathbf{y}}|^2}\omega(\mathbf{y})d\sigma(\mathbf{y})+O(\|\tilde{\omega}\|_{L^\infty})\\
        &=\int_{Q\left(\frac{3\varphi}{2},\frac{3\theta}{2}\right)}\frac{x_1y_2y_3}{|\mathbf{y}-\mathbf{x}_0|^4}\omega(\mathbf{y})d\sigma(\mathbf{y})+O(\|\tilde{\omega}\|_{L^\infty}),
    \end{split}        
    \end{equation}
    and
    \begin{equation}\label{Q1f2'}
    \begin{split}
        \int_{Q\left(\frac{3\varphi}{2},\frac{3\theta}{2}\right)}\frac{y_3\big((x_2+y_2)x_1-(x_1-y_1)x_2\big)}{|\mathbf{x}-\tilde{\mathbf{y}}|^2|\mathbf{x}-\tilde{\bar{\mathbf{y}}}|^2}\omega(\mathbf{y})d\sigma(\mathbf{y})&=\int_{Q\left(\frac{3\varphi}{2},\frac{3\theta}{2}\right)}\frac{x_1y_2y_3}{|\mathbf{x}-\tilde{\mathbf{y}}|^2|\mathbf{x}-\tilde{\bar{\mathbf{y}}}|^2}\omega(\mathbf{y})d\sigma(\mathbf{y})+O(\|\tilde{\omega}\|_{L^\infty})\\
        &=\int_{Q\left(\frac{3\varphi}{2},\frac{3\theta}{2}\right)}\frac{x_1y_2y_3}{|\mathbf{y}-\mathbf{x}_0|^4}\omega(\mathbf{y})d\sigma(\mathbf{y})+O(\|\tilde{\omega}\|_{L^\infty}).
    \end{split}        
    \end{equation}
    Next, consider the integral over $\tilde{Q}_2$. With the change of variables $(\varphi'',\theta'')=(\varphi'-\varphi,\theta'-\theta)$, we obtain
    \begin{equation*} 
        \begin{split}
            \int_{\tilde{Q}_2}\frac{y_3\big((x_1-y_1)y_2-(x_2-y_2)y_1\big)}{|\mathbf{x}-{\mathbf{y}}|^2|\mathbf{x}-{\bar{\mathbf{y}}}|^2}\omega(\mathbf{y})d\sigma(\mathbf{y})&\leq\|\tilde{\omega}\|_{L^\infty}\int_{\tilde{Q}_2}\frac{x_3(|x_1-y_1|+|x_2-y_2|)}{|\mathbf{x}-{\mathbf{y}}|^2|\mathbf{x}-{\bar{\mathbf{y}}}|^2}d\sigma(\mathbf{y})\\
            &\lesssim\|\tilde{\omega}\|_{L^\infty}\int^\pi_0\int^\theta_0\frac{\varphi''\sin\theta}{(\varphi''^{\mkern 1mu 2}+\theta''^{\mkern 1mu 2})\big((\sin\theta)^2+\varphi''^{\mkern 1mu 2}\big)}d\theta''d\varphi''\\
            &\lesssim\|\tilde{\omega}\|_{L^\infty}\int^\pi_0\frac{\theta\,\varphi''}{\varphi''(\theta^2+\varphi''^{\mkern 1mu 2})}\arctan\frac{\theta}{\varphi''}d\varphi''\\
            &\lesssim\|\tilde{\omega}\|_{L^\infty}.
        \end{split}
    \end{equation*}
    Then, letting $(\varphi'',\theta'')=(\varphi'+\varphi,\theta'-\theta)$, we get
    \begin{equation*} 
        \begin{split}
            \int_{\tilde{Q}_2}\frac{y_3\big((x_2+y_2)x_1-(x_1-y_1)x_2\big)}{|\mathbf{x}-\tilde{\mathbf{y}}|^2|\mathbf{x}-\tilde{\bar{\mathbf{y}}}|^2}\omega(\mathbf{y})d\sigma(\mathbf{y})&\lesssim\|\tilde{\omega}\|_{L^\infty}\int_{\tilde{Q}_2}\frac{x_3((y_2+x_2)+|y_1-x_1|)}{|\mathbf{x}-\tilde{\mathbf{y}}|^2|\mathbf{x}-\tilde{\bar{\mathbf{y}}}|^2}d\sigma(\mathbf{y})\\           &\lesssim\|\tilde{\omega}\|_{L^\infty}\int^{2\pi}_0\int^\theta_0\frac{\varphi''\sin\theta}{(\varphi''^{\mkern 1mu 2}+\theta''^{\mkern 1mu 2})\big((\sin\theta)^2+\varphi''^{\mkern 1mu 2}\big)}d\theta''d\varphi''\\
            &\lesssim\|\tilde{\omega}\|_{L^\infty}.
        \end{split}
    \end{equation*}
    For the integral over $\tilde{Q}_3$, we again employ the change $(\varphi'',\theta'')=(\varphi'-\varphi,\theta'-\theta)$ to obtain
    \begin{equation}\label{Q3f1'}
        \begin{split}
            \int_{\tilde{Q}_3}\frac{y_3\big((x_1-y_1)y_2-(x_2-y_2)y_1\big)}{|\mathbf{x}-{\mathbf{y}}|^2|\mathbf{x}-{\bar{\mathbf{y}}}|^2}\omega(\mathbf{y})d\sigma(\mathbf{y})&\lesssim\|\tilde{\omega}\|_{L^\infty}\int^\varphi_0\int^\frac{\pi}{2}_\frac{\theta}{2}\frac{\varphi''\theta''}{(\varphi''^{\mkern 1mu 2}+\theta''^{\mkern 1mu 2})^2}d\theta''d\varphi''\\
            &\lesssim\|\tilde{\omega}\|_{L^\infty}\int^\varphi_0\frac{4\varphi''}{\theta^2+4\varphi''^{\mkern 1mu 2}}d\varphi''\\
            &\lesssim\|\tilde{\omega}\|_{L^\infty}\left(\ln\left(1+\frac{\varphi}{\theta}\right)+1\right).
        \end{split}
    \end{equation}
    When $\varphi\leq\theta$, we have $\ln(1+\varphi/\theta)\leq\ln2$, so the bound is of order $\|\tilde{\omega}\|_{L^\infty}$. If instead $\varphi>\theta$, we further split $\tilde{Q}_3$ into
    \begin{equation*}
        \begin{split}
            &\tilde{Q}_3'=\left\{(\cos\theta'\cos\varphi',\cos\theta'\sin\varphi',\sin\theta'):\varphi'\in\left(0,\frac{3\varphi}{2}\right],\,\theta'\in\left[\frac{3\varphi}{2},\frac{\pi}{2}\right)\right\},\\
            &\tilde{Q}_4=\left\{(\cos\theta'\cos\varphi',\cos\theta'\sin\varphi',\sin\theta'):\varphi'\in\left(0,\frac{3\varphi}{2}\right],\,\theta'\in\left[\frac{3\theta}{2},\frac{3\varphi}{2}\right]\right\}.
        \end{split}
    \end{equation*}
    Following the computation in \eqref{Q3f1'}, we obtain
    \begin{equation}\label{Q3'f1'}
        \int_{\tilde{Q}_3'}\frac{y_3\big((x_1-y_1)y_2-(x_2-y_2)y_1\big)}{|\mathbf{x}-{\mathbf{y}}|^2|\mathbf{x}-{\bar{\mathbf{y}}}|^2}\omega(\mathbf{y})d\sigma(\mathbf{y})\lesssim\|\tilde{\omega}\|_{L^\infty}(1+\ln2).
    \end{equation}
    In $\tilde{Q}_4$, set $M':=\|\nabla\tilde{\omega}\|_{L^\infty([0,3\varphi/2]^2)}$ and decompose $\omega$ as
    \begin{equation*}
        \omega(\mathbf{y})=\omega_3(y_3)+\omega_4(\mathbf{y}),
    \end{equation*}
    where 
    \begin{equation*}
        \omega_3(y_3)=\omega\left(x_1\frac{\sqrt{1-y_3^2}}{\sqrt{1-x_3^2}},x_2\frac{\sqrt{1-y_3^2}}{\sqrt{1-x_3^2}},y_3\right),
    \end{equation*}
    so that $\tilde{\omega}_3(\theta')=\tilde{\omega}(\varphi,\theta')$. Then we have
    \begin{equation*}
        |\omega_4(\mathbf{y})|=|\omega(\mathbf{y})-\omega_3(y_3)|\leq M'|\varphi'-\varphi|.
    \end{equation*}
    By odd symmetry,
    \begin{equation*}
        |\omega_3(y_3)|=|\omega_3(y_3)-\omega_3(0)|\leq M'|\theta'|.
    \end{equation*}
    Setting $\varphi''=\varphi'-\varphi$, then we obtain
    \begin{equation}\label{Q4w1'}
        \begin{split}
            \int_{\tilde{Q}_4}\frac{y_3\big((x_1-y_1)y_2-(x_2-y_2)y_1\big)}{|\mathbf{x}-{\mathbf{y}}|^2|\mathbf{x}-{\bar{\mathbf{y}}}|^2}\omega_3(y_3)d\sigma(\mathbf{y})&\leq\int_{\tilde{Q}_4}\frac{y_3(|x_1-y_1|+|x_2-y_2|)}{|\mathbf{x}-{\mathbf{y}}|^2|\mathbf{x}-{\bar{\mathbf{y}}}|^2}|\omega_3(y_3)|d\sigma(\mathbf{y})\\
            &\lesssim\int^\varphi_0\int^\frac{3\varphi}{2}_\frac{3\theta}{2}\frac{M'\theta'}{\varphi''^{\mkern 1mu 2}+\theta'^{\mkern 2mu 2}}d\theta''d\varphi'\\
            &\lesssim M'\varphi,
        \end{split}
    \end{equation}
    and
    \begin{equation}\label{Q4w2'}
    \begin{split}
        \int_{\tilde{Q}_4}\frac{y_3\big((x_1-y_1)y_2-(x_2-y_2)y_1\big)}{|\mathbf{x}-{\mathbf{y}}|^2|\mathbf{x}-{\bar{\mathbf{y}}}|^2}\omega_4(\mathbf{y})d\sigma(\mathbf{y})&\lesssim M'\int^\varphi_0\int^\frac{3\varphi}{2}_\frac{3\theta}{2}\frac{\theta'\varphi''}{(\theta'^{\mkern 2mu 2}+\varphi''^{\mkern 1mu 2})^2}d\theta'd\varphi''\\
        &\lesssim M'\varphi.
    \end{split}        
    \end{equation}
    For the second term in the integral on the right-hand side of \eqref{utheta}, when $\varphi>\theta$ we can still integrate over $\tilde{Q}_3$ without splitting. Using the change $(\varphi'',\theta'')=(\varphi'+\varphi,\theta'-\theta)$, we get
    \begin{equation}\label{Q3'f2'}
        \begin{split}
            \int_{\tilde{Q}_3}\frac{y_3\big((x_2+y_2)x_1-(x_1-y_1)x_2\big)}{|\mathbf{x}-\tilde{\mathbf{y}}|^2|\mathbf{x}-\tilde{\bar{\mathbf{y}}}|^2}\omega(\mathbf{y})d\sigma(\mathbf{y})&\lesssim\|\tilde{\omega}\|_{L^\infty}\int^{\frac{\pi}{2}}_{\frac{\theta}{2}}\int^\frac{5\varphi}{2}_\varphi\frac{\varphi''\theta''}{(\varphi''^{\mkern 1mu 2}+\theta''^{\mkern 1mu 2})^2}d\varphi''d\theta''\\
            &\lesssim\|\tilde{\omega}\|_{L^\infty}.
        \end{split}
    \end{equation}
    Finally, combining \eqref{Q1f1'}–\eqref{Q3'f2'}, we arrive at \eqref{velocity estimate2}.
\end{proof}

\begin{remark}
    The region $Q(3\varphi/2,3\theta/2)$ in Lemma \ref{Velocity estimate} can be replaced by 
    \begin{equation*}
        \tilde{Q}\Big(\sqrt{\varphi^2+\theta^2}\Big):=\left\{(\cos\theta'\cos\varphi',\cos\theta'\sin\varphi',\sin\theta'):\varphi'\in\left(0,\pi\right),\,\theta'\in\left(0,\frac{\pi}{2}\right),\,\sqrt{\varphi'^{\mkern 2mu 2}+\theta'^{\mkern 2mu 2}}>\sqrt{\varphi^2+\theta^2}\right\},
    \end{equation*}
    because the difference between these two domains is contained in $D_1\cup D_2$, where
    \begin{equation*}
        \begin{split}
            &D_1=\left\{(\cos\theta'\cos\varphi',\cos\theta'\sin\varphi',\sin\theta'):\varphi'\in\left(0,\min\left\{\pi,\frac{3}{2}\sqrt{\varphi^2+\theta^2}\right\}\right],\,\theta'\in\left[\frac{1}{2}\sqrt{\varphi^2+\theta^2},\frac{\pi}{2}\right)\right\},\\
            &D_2=\left\{(\cos\theta'\cos\varphi',\cos\theta'\sin\varphi',\sin\theta'):\varphi'\in\left[\frac{1}{2}\sqrt{\varphi^2+\theta^2},\pi\right),\,\theta'\in\left(0,\min\left\{\frac{\pi}{2},\frac{3}{2}\sqrt{\varphi^2+\theta^2}\right\}\right]\right\}.
        \end{split}
    \end{equation*}
    A direct computation then yields
    \begin{equation*}
        \begin{split}
            \int_{D_1}\frac{y_2y_3}{|\mathbf{y}-\mathbf{x}_0|^4}\omega(\mathbf{y})d\sigma(\mathbf{y})&\lesssim\|\tilde{\omega}\|_{L^\infty}\int^\frac{\pi}{2}_\frac{\sqrt{\varphi^2+\theta^2}}{2}\int^{\min\left\{\pi,\frac{3}{2}\sqrt{\varphi^2+\theta^2}\right\}}_0\frac{\varphi'\cos\theta'}{(\varphi'^{\mkern 2mu 2}+\theta'^{\mkern 2mu 2})^\frac{3}{2}}d\varphi'd\theta'\\
            &\lesssim\|\tilde{\omega}\|_{L^\infty},
        \end{split}
    \end{equation*}
    \begin{equation*}
        \begin{split}
            \int_{D_2}\frac{y_2y_3}{|\mathbf{y}-\mathbf{x}_0|^4}\omega(\mathbf{y})d\sigma(\mathbf{y})&\lesssim\|\tilde{\omega}\|_{L^\infty}\int^\pi_\frac{\sqrt{\varphi^2+\theta^2}}{2}\int^{\min\left\{\frac{\pi}{2},\frac{3}{2}\sqrt{\varphi^2+\theta^2}\right\}}_0\frac{\theta'\cos\theta'}{(\varphi'^{\mkern 2mu 2}+\theta'^{\mkern 2mu 2})^\frac{3}{2}}d\theta'd\varphi'\\
            &\lesssim\|\tilde{\omega}\|_{L^\infty}.
        \end{split}
    \end{equation*}
    Hence the replacement does not affect the leading-order terms.
\end{remark}

\subsection{Upper bound control and proof of Theorem \ref{Main2}}

In this subsection we complete the proof of Theorem~\ref{Main2}. We first estimate the speed at which two points on $\mathbb{S}^2$ approach each other. For a general vorticity $\tilde{\omega}\in W^{1,\infty}([-\pi,\pi)\times[-\pi/2,\pi/2])$ with $\|\omega\|_{L^\infty(\mathbb{S}^2)}=1$, we do not assume the odd-odd symmetry condition.

Let $\varepsilon\in(0,1]$ and consider two points on $\mathbb{S}^2$ whose geodesic distance is $2\varepsilon$. By rotational symmetry we may, without loss of generality, take these points to be $\mathbf{x}_{1,\varepsilon}=(\cos\varepsilon,-\sin\varepsilon,0)$ and $\mathbf{x}_{2,\varepsilon}=(\cos\varepsilon,\sin\varepsilon,0)$. Their relative approach speed on the sphere is then given by
\begin{equation}\label{approach v}
    \begin{split}
        \tilde{u}_\varphi(-\varepsilon,0)-\tilde{u}_\varphi(\varepsilon,0)&=\frac{1}{2\pi}\int_{\mathbb{S}^2}\frac{-y_3\omega(\mathbf{y})}{|\mathbf{x}_{1,\varepsilon}-\mathbf{y}|^2}d\sigma(\mathbf{y})-\frac{1}{2\pi}\int_{\mathbb{S}^2}\frac{-y_3\omega(\mathbf{y})}{|\mathbf{x}_{2,\varepsilon}-\mathbf{y}|^2}d\sigma(\mathbf{y})\\
        &=\frac{2\sin\varepsilon}{\pi}\int_{\mathbb{S}^2}\frac{y_2y_3\omega(\mathbf{y})}{|\mathbf{x}_{1,\varepsilon}-\mathbf{y}|^2|\mathbf{x}_{2,\varepsilon}-\mathbf{y}|^2}d\sigma(\mathbf{y}).
    \end{split}
\end{equation}

Assume that $\|\omega\|_{L^1(\mathbb{S}^2)}=I\in(0,4\pi]$. Then the right‑hand side of \eqref{approach v} attains its maximum when
\[
\omega(\mathbf{y})=\omega_\varepsilon(\mathbf{y}):=\operatorname{sgn}(y_2y_3)\,\mathcal{X}_{L_\varepsilon(a_{I,\varepsilon})}(\mathbf{y}),
\]
where $\operatorname{sgn}$ is the sign function, $\mathcal{X}_A$ is the characteristic function of a set $A$, and
\[
L_\varepsilon(a_{I,\varepsilon}):=\left\{\mathbf{y}\in\mathbb{S}^2:\frac{|y_2y_3|}{|\mathbf{x}_{1,\varepsilon}-\mathbf{y}|^2|\mathbf{x}_{2,\varepsilon}-\mathbf{y}|^2}\geq a_{I,\varepsilon}\right\},
\]
with $a_{I,\varepsilon}>0$ chosen so that $\|\mathcal{X}_{L_\varepsilon(a_{I,\varepsilon})}\|_{L^1(\mathbb{S}^2)}=I$. Denote by $u^*(\mathbf{x})=\tilde{u}^*_\varphi(\varphi,\theta)\mathbf{e}_\varphi+\tilde{u}^*_\theta(\varphi,\theta)\mathbf{e}_\theta$ the velocity at $\mathbf{x}\in\mathbb{S}^2$ associated with $\omega_\varepsilon$. Substituting $\omega_\varepsilon$ into \eqref{approach v} and noting that $\omega_\varepsilon$ is odd‑odd symmetric, we obtain
\begin{equation*}
    \begin{split}
        |\tilde{u}_\varphi(-\varepsilon,0)-\tilde{u}_\varphi(\varepsilon,0)|
        &\leq \tilde{u}^*_\varphi(-\varepsilon,0)-\tilde{u}^*_\varphi(\varepsilon,0)\\
        &=\frac{8\sin\varepsilon}{\pi}\int_{\mathbb{S}^2_{q}\cap L_\varepsilon(a_{I,\varepsilon})}\frac{y_2y_3}{|\mathbf{x}_{1,\varepsilon}-\mathbf{y}|^2|\mathbf{x}_{2,\varepsilon}-\mathbf{y}|^2}d\sigma(\mathbf{y}).
    \end{split}
\end{equation*}

By Lemma~\ref{Velocity estimate} and the subsequent remark, we have
\begin{equation*}
    \left|\tilde{u}^*_\varphi(-\varepsilon,0)-\tilde{u}^*_\varphi(\varepsilon,0)-\frac{8\sin\varepsilon}{\pi}\int_{\tilde{Q}(\varepsilon)\cap L_\varepsilon(a_{I,\varepsilon})}\frac{y_2y_3}{|\mathbf{y}-\mathbf{x}_0|^4}d\sigma(\mathbf{y})\right|\leq 2C_1\varepsilon,
\end{equation*}
and consequently
\begin{equation*}
    \tilde{u}^*_\varphi(-\varepsilon,0)-\tilde{u}^*_\varphi(\varepsilon,0)\leq\frac{8\sin\varepsilon}{\pi}\int_{\tilde{Q}(\varepsilon)\cap L_\varepsilon(a_{I,\varepsilon})}\frac{y_2y_3}{|\mathbf{y}-\mathbf{x}_0|^4}d\sigma(\mathbf{y})+2C_1\varepsilon.
\end{equation*}

For the integral on the right‑hand side of the above inequality we note that
\begin{equation*}
    \frac{y_2y_3}{|\mathbf{y}-\mathbf{x}_0|^4}\cos\theta'=\frac{\varphi'\theta'}{(\varphi'^{\mkern 2mu 2}+\theta'^{\mkern 2mu 2})^2}+O(1)\quad \text{for}\quad \mathbf{y}\in \mathbb{S}^2\setminus\{\mathbf{x}_0\}.
\end{equation*}
 Moreover, applying the polar coordinate transformation $(\varphi',\theta')=(r\cos\theta'',r\sin\theta'')$ yields
\begin{equation*}
\begin{split}
    \int_{\tilde{Q}(\varepsilon)\cap L_\varepsilon(a_{I,\varepsilon})}\frac{y_2y_3}{|\mathbf{y}-\mathbf{x}_0|^4}d\sigma(\mathbf{y})
    &\leq\int^{2\pi}_\varepsilon\int^\frac{\pi}{2}_0\frac{\cos\theta''\sin\theta''}{r}\,d\theta''dr+C\\
    &\leq-\frac{1}{2}\ln\varepsilon+C,
\end{split}    
\end{equation*}
which implies that there exists a constant $C_1'>0$ such that
\begin{equation}\label{v 1}
    |\tilde{u}_\varphi(-\varepsilon,0)-\tilde{u}_\varphi(\varepsilon,0)|
    \leq\tilde{u}^*_\varphi(-\varepsilon,0)-\tilde{u}^*_\varphi(\varepsilon,0)
    \leq\frac{2}{\pi}\left(\ln\frac{1}{\varepsilon}+C_1'\right)2\varepsilon.
\end{equation}

On the other hand, if $\|\nabla\tilde{\omega}\|_{L^\infty}$ is bounded, then we have
\begin{equation*}
    |\tilde{\omega}(\varphi',\theta')-\tilde{\omega}(0,\theta')|\leq \|\nabla\tilde{\omega}\|_{L^\infty}|\varphi'|.
\end{equation*}
Set $m:=\min\{1/\|\nabla\tilde{\omega}\|_{L^\infty},1\}$. Proceeding analogously to the computation leading to \eqref{Q1f1}, for $\mathbf{y}\in\mathbb{S}^2_{q}\cap\tilde{Q}(2\varepsilon)$ we obtain
\begin{equation*}
    \begin{split}
        \frac{y_2y_3}{|\mathbf{x}_{1,\varepsilon}-\mathbf{y}|^2|\mathbf{x}_{2,\varepsilon}-\mathbf{y}|^2}-\frac{y_2y_3}{|\mathbf{y}-\mathbf{x}_0|^4}
        =&O\left(\frac{|\mathbf{y}-\mathbf{x}_0|^4-|(\mathbf{x}_{1,\varepsilon}-\mathbf{x}_0)-(\mathbf{y}-\mathbf{x}_0)|^2|(\mathbf{x}_{2,\varepsilon}-\mathbf{x}_0)-({\mathbf{y}}-\mathbf{x}_0)|^2}{|\mathbf{y}-\mathbf{x}_0|^6}\right)\\
        =&O\left(\frac{|\mathbf{x}_{1,\varepsilon}-\mathbf{x}_0||\mathbf{y}-\mathbf{x}_0|^3}{|\mathbf{y}-\mathbf{x}_0|^6}\right)\\
        =&O\left(\frac{|\mathbf{x}_{1,\varepsilon}-\mathbf{x}_0|}{|\mathbf{y}-\mathbf{x}_0|^3}\right).
    \end{split}        
\end{equation*}
A direct computation then gives
\begin{equation*}
    \begin{split}
         \int_{\mathbb{S}^2_{q}\cap\tilde{Q}(2\varepsilon)}\frac{|\mathbf{x}_{1,\varepsilon}-\mathbf{x}_0|}{|\mathbf{y}-\mathbf{x}_0|^3}d\sigma(\mathbf{y})
         &\lesssim\sin\varepsilon\int^{\pi}_\varepsilon\frac{1}{r^2}\,dr\\
         &\lesssim1.
    \end{split}
\end{equation*}

For sufficiently small $\varepsilon$, we also have
\begin{equation*}
\begin{split}
     \int_{\mathbb{S}^2_{q}\setminus\tilde{Q}(2\varepsilon)}\frac{y_2y_3}{|\mathbf{x}_{1,\varepsilon}-\mathbf{y}|^2|\mathbf{x}_{2,\varepsilon}-\mathbf{y}|^2}\min\{\|\nabla\tilde{\omega}\|_{L^\infty}\,\varphi',1\}\,d\sigma(\mathbf{y})
     &\lesssim\int_{\mathbb{S}^2_{q}\setminus\tilde{Q}(2\varepsilon)}\frac{1}{|\mathbf{x}_{1,\varepsilon}-\mathbf{y}||\mathbf{x}_{2,\varepsilon}-\mathbf{y}|}d\sigma(\mathbf{y})\\
     &\lesssim\frac{1}{\sin\varepsilon}\int^{2\varepsilon}_01\,dr\\
     &\lesssim1.
\end{split}
\end{equation*}

Combining the estimates above, we obtain
\begin{equation*}
\begin{split}
     |\tilde{u}_\varphi(-\varepsilon,0)-\tilde{u}_\varphi(\varepsilon,0)|
     &\leq\frac{8\sin\varepsilon}{\pi}\int_{\mathbb{S}^2_{q}}\frac{y_2y_3}{|\mathbf{x}_{1,\varepsilon}-\mathbf{y}|^2|\mathbf{x}_{2,\varepsilon}-\mathbf{y}|^2}\min\{\|\nabla\tilde{\omega}\|_{L^\infty}\,\varphi',1\}\,d\sigma(\mathbf{y})\\
     &=\frac{8\sin\varepsilon}{\pi}\int_{\mathbb{S}^2_{q}\cap\tilde{Q}(2\varepsilon)}\frac{y_2y_3}{|\mathbf{x}_{1,\varepsilon}-\mathbf{y}|^2|\mathbf{x}_{2,\varepsilon}-\mathbf{y}|^2}\min\{\|\nabla\tilde{\omega}\|_{L^\infty}\,\varphi',1\}\,d\sigma(\mathbf{y})+O(\varepsilon)\\
     &\leq\frac{8\varepsilon}{\pi}\int_{\mathbb{S}^2_{q}\cap\tilde{Q}(2\varepsilon)}\frac{y_2y_3}{|\mathbf{y}-\mathbf{x}_0|^4}\min\{\|\nabla\tilde{\omega}\|_{L^\infty}\,\varphi',1\}\,d\sigma(\mathbf{y})+C\varepsilon.
\end{split}   
\end{equation*}
Applying polar coordinates once more and treating the cases $y\in\tilde{Q}(m)$ and $y\notin\tilde{Q}(m)$ separately, we get
\begin{equation}\label{v 2}
     \begin{split}
         |\tilde{u}_\varphi(-\varepsilon,0)-\tilde{u}_\varphi(\varepsilon,0)|
         &\leq\frac{8\varepsilon}{\pi}\int^1_m\int^\frac{\pi}{2}_0\frac{\cos\theta''\sin\theta''}{r}\,d\theta''dr
          +\frac{8\varepsilon\|\nabla\tilde{\omega}\|_{L^\infty}}{\pi}\int^m_{2\varepsilon}\int^\frac{\pi}{2}_01\,d\theta''dr+C\varepsilon\\
         &\leq\left(\frac{4}{\pi}\ln \frac{1}{m}+C\right)\varepsilon.
     \end{split}
 \end{equation}

Now combine \eqref{v 1} and \eqref{v 2} and denote $\delta=2\varepsilon$. Then there exists a constant $C_I>0$ such that for any two points on $\mathbb{S}^2$ whose geodesic distance is $\delta$, the speed at which they approach each other along $\mathbb{S}^2$ is bounded by
\begin{equation}\label{v 3}
     \frac{2}{\pi}\left(\ln\min\left\{\frac{2}{\delta},\max\{\|\nabla\tilde{\omega}\|_{L^\infty},1\}\right\}+C_I\right)\delta.
\end{equation}

Consequently, for any $\tilde{\omega}(\cdot,0)\in W^{1,\infty}([-\pi,\pi)\times[-\pi/2,\pi/2])$ with $\|{\omega}(\cdot,0)\|_{L^\infty(\mathbb{S}^2)}=1$ and $\|{\omega}(\cdot,0)\|_{L^1(\mathbb{S}^2)}=I\in(0,4\pi]$, the following estimate holds:
\begin{equation}
     \|\nabla\tilde{\omega}(\cdot,t)\|_{L^\infty}\leq\exp\left(e^{2t/\pi+C_I(1-e^{-2t/\pi})}\max\{\ln\|\nabla\tilde{\omega}(\cdot,0)\|_{L^\infty},1\}\right)\quad\text{for any}\quad t\geq0,
\end{equation}
which implies \eqref{main2}.

\section{Proof of Theorem \ref{Main1}}
This section is devoted to the proof of Theorem~\ref{Main1}. The main goal is to construct a suitable initial vorticity function and to find an appropriate domain in the quarter-sphere such that the vorticity at each point in this domain has a positive lower bound (or simply equals $1$), and the distance from this domain to $\{\varphi=0\}$ shrinks at a double‑exponential rate. Moreover, to achieve the maximal growth rate $\pi/2$, the domain needs to approximately fill the shrinking neighborhood of $\mathbf{x}_0$ on $\mathbb{S}^2_{q}$.

\begin{figure}[htbp]
    \centering
    \includegraphics[width=0.5\textwidth]{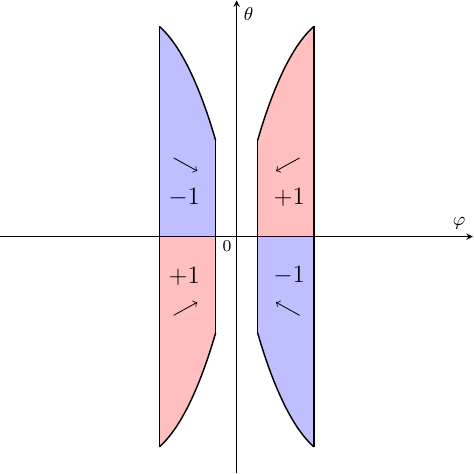}
    \caption{Four domains in the $(\varphi,\theta)$-plane whose distances to the $\theta$-axis (where $\tilde{\omega}=0$ by odd symmetry) shrink rapidly; inside each domain $|\tilde{\omega}|=1$.}
    \label{fig:2}
\end{figure}

Let $g(s):=s\ln(e^e-1+|\ln s|)$ for $s\geq0$, and for $\varepsilon\in[0,1/e)$, define
\begin{equation*}
    \Omega_\varepsilon:=\{(\varphi',\theta')\in(0,1)^2:\varphi'\in(\varepsilon,e^{-1}),\ \theta'<g(\varphi')\}.
\end{equation*}
Denote
\begin{equation*}
    D_s:=\left\{(\varphi',\theta')\in\Omega_0:s^2+g^2(s)<\varphi'^{\mkern 2mu 2}+\theta'^{\mkern 2mu 2}<{e^{-2}}\right\}
\end{equation*}
and $\hat{D}_s:=\{(\cos\theta'\cos\varphi',\cos\theta'\sin\varphi',\sin\theta'):(\varphi',\theta')\in D_s\}$.
Expanding $y_2y_3\cos\theta'$ and $|\mathbf{y}-\mathbf{x}_0|^4$ then yields
\begin{equation}
\begin{split}
    \frac{4}{\pi}\int_{\hat{D}_s}\frac{y_2y_3}{|\mathbf{y}-\mathbf{x}_0|^4}d\sigma(\mathbf{y})&=\frac{4}{\pi}\int_{D_s}\left(\frac{\varphi'\theta'}{(\varphi'^{\mkern 2mu 2}+\theta'^{\mkern 2mu 2})^2}+O(1)\right)d\varphi'd\theta'\\
    &=h(s)|\ln s|+O(1),
\end{split}    
\end{equation}
where $h:(0,1/e^4]\to(0,\infty)$ satisfies
\begin{equation}
    \lim_{s\to0}h(s)=\frac{2}{\pi}.
\end{equation}
This limit follows from a direct computation giving
\begin{equation*}
    \lim_{s\to0}\frac{|\ln g(s)|}{|\ln s|}=\lim_{s\to0}\frac{-\ln s-\ln\ln(e^e-1-\ln s)}{-\ln s}=1,
\end{equation*}
$\lim_{s\to0}\frac{g(s)}{s}=\infty$, and 
\begin{equation*}
    \frac{4}{\pi}\int^{\frac{\pi}{2}}_0\cos\theta''\sin\theta''\,d\theta''=\frac{2}{\pi},
\end{equation*}
which is similar to the calculation in \cite{ZA2}.
Then there exists a constant $C_r>0$ such that
\begin{equation*}
    h(s)|\ln s|-C_r\leq\frac{4}{\pi}\int_{\hat{D}_s}\frac{y_2y_3}{|\mathbf{y}-\mathbf{x}_0|^4}d\sigma(\mathbf{y})\leq h(s)|\ln s|+C_r.
\end{equation*}
Denote 
\begin{equation}\label{f}
    f(s):=2+\ln\ln(e^e-1+|\ln s|)\geq1+\ln\big(1+\ln(e^e-1+|\ln s|)\big)=1+\ln\frac{s+g(s)}{s},\qquad \forall s\in \left(0, e^{-1}\right),
\end{equation}
and let $s_0\in(0,1/e^4]$ be such that
\begin{equation}\label{h,f}
    h(s)|\ln s|\geq\frac{C_1}{\cos1\,\cos e^{-1}}\,f(s)+C_r,\qquad \forall s\in (0, s_0].
\end{equation}
Note that for $s\in(0,e^{-1}]$, we have
\begin{equation}\label{g'(s)}
    g'(s)=\ln(e^e-1-\ln s)-\frac{1}{e^e-1-\ln s}\in\left(1,\frac{g(s)}{s}\right),
\end{equation}
and $g(e^{-1})=1$.
Then define
\begin{equation*}
    \begin{split}
        C':=&\frac{g(s_0)C_1}{s_0\cos1}\left(1+\ln\frac{s_0+g(s_0)}{s_0}\right)=\max_{s\in\left[s_0,{e}^{-1}\right)}\frac{g(s)C_1}{s\cos1}\left(1+\ln\frac{s+g(s)}{s}\right),\\
        &\rho_0:=\frac{s_0}{(e^e-1-\ln s_0)\ln(e^e-1-\ln s_0)-1}=\min_{s\in\left[s_0,{e}^{-1}\right)}\left(\frac{g(s)}{g'(s)}-s\right).
    \end{split}
\end{equation*}

Let $\varepsilon\in(0,s_0]$. For any $\tilde{\omega}\in L^\infty([-\pi,\pi)\times[-\pi/2,\pi/2])$ that is odd‑odd symmetric and satisfies
\begin{equation*}
    \mathcal{X}_{\Omega_\varepsilon}\leq\tilde{\omega}\,\mathcal{X}_{\left([0,\pi)\times\left[0,\frac{\pi}{2}\right]\right)}\leq\mathcal{X}_{\left([0,\pi)\times\left[0,\frac{\pi}{2}\right]\right)},
\end{equation*}
by Lemma~\ref{Velocity estimate} and the subsequent remark, together with \eqref{f} and \eqref{h,f}, it follows that for any $s\in[\varepsilon,s_0]$,
\begin{equation*}
    \begin{split}
        \frac{4}{\pi}\int_{\tilde{Q}\big(\sqrt{s^2+g^2(s)}\,\big)}\frac{y_2y_3}{|\mathbf{y}-\mathbf{x}_0|^4}\omega(\mathbf{y})d\sigma(\mathbf{y})&\geq\frac{4}{\pi}\int_{\hat{D}_s}\frac{y_2y_3}{|\mathbf{y}-\mathbf{x}_0|^4}d\sigma(\mathbf{y})\\
        &\geq \frac{C_1}{\cos 1\,\cos e^{-1}}\,f(s)\\
        &\geq\max\left\{|B_\varphi(s,g(s))|,\ \frac{|B_\theta(s,g(s))|}{\cos g(s)\,\cos s}\right\}.
    \end{split}
\end{equation*}
Thus we obtain
\begin{equation*}
    \tilde{u}_\varphi(s,g(s))\leq0,\quad \tilde{u}_\theta(s,g(s))\geq0\quad\text{for any}\quad s\in[\varepsilon,s_0].
\end{equation*}
For $s\in[s_0,e^{-1})$, notice that
\begin{equation*}
   \begin{split}
       \tilde{u}_\varphi(s,g(s))&\leq|B_\varphi(s,g(s))|\sin s\leq C_1\left(1+\ln\frac{s+g(s)}{s}\right)s,
   \end{split}
\end{equation*}
and
\begin{equation*}
    \begin{split}
        \tilde{u}_\theta(s,g(s))&\geq-\frac{\sin g(s)}{\cos g(s)}|B_\theta(s,g(s))|\\        
        &\geq-\frac{g(s)C_1}{s\cos1}\left(1+\ln\frac{s+g(s)}{g(s)}\right)s.
    \end{split}
\end{equation*}
Then, from the definition of $C'$, we have
\begin{equation*}
    \tilde{u}_\varphi(s,g(s))<C's,\quad \tilde{u}_\theta(s,g(s))>-C's\quad\text{for any}\quad s\in\left[s_0,e^{-1}\right).
\end{equation*}
Furthermore, for any $s_1\in[0,g(\varepsilon)]$ and $s_2\in[0,1]$, using the definitions of $\hat{D}_\varepsilon$ and $f(\varepsilon)$, together with $1+\ln(1+es_2)<3$ for any $s_2\in[0,1]$, we obtain
\begin{equation*}
    \begin{split}
        \frac{1}{\sin\varepsilon}\tilde{u}_\varphi(\varepsilon,s_1)-\frac{1}{\sin e^{-1}}\tilde{u}_\varphi(e^{-1},s_2)\leq&-\frac{4}{\pi}\int_{\tilde{Q}\big(\sqrt{\varepsilon^2+s_1^2}\,\big) }\frac{y_2y_3}{|\mathbf{y}-\mathbf{x}_0|^4}\omega(\mathbf{y})d\sigma(\mathbf{y})+|B_\varphi(\varepsilon,s_1)|\\
        &+\frac{4}{\pi}\int_{\tilde{Q}\big(\sqrt{e^{-2}+s_2^2}\,\big) }\frac{y_2y_3}{|\mathbf{y}-\mathbf{x}_0|^4}\omega(\mathbf{y})d\sigma(\mathbf{y})+|B_\varphi(e^{-1},s_2)|\\
        \leq&-\frac{4}{\pi}\int_{\hat{D}_\varepsilon}\frac{y_2y_3}{|\mathbf{y}-\mathbf{x}_0|^4}d\sigma(\mathbf{y})+C_1f(\varepsilon)+3C_1.
    \end{split}
\end{equation*}
This implies that there exists some $C_2>3C_1+C_r$ such that
\begin{equation*}
    \frac{1}{\sin\varepsilon}\sup_{s\in[0,g(\varepsilon)]}\tilde{u}_\varphi(\varepsilon,s)\leq-h(\varepsilon)|\ln\varepsilon|+C_1f(\varepsilon)+C_2+\frac{1}{\sin e^{-1}}\inf_{s\in[0,1]}\tilde{u}_\varphi(e^{-1},s).
\end{equation*}
For any $\alpha\in(0,1]$, replace $D_s$ by $\alpha D_s$ and repeat the above discussion. Note that
\begin{equation*}
    \int_{\alpha D_s}\frac{\varphi'\theta'}{(\varphi'^{\mkern 2mu 2}+\theta'^{\mkern 2mu 2})^2}d\varphi'd\theta'=\int_{ D_s}\frac{\varphi'\theta'}{(\varphi'^{\mkern 2mu 2}+\theta'^{\mkern 2mu 2})^2}d\varphi'd\theta'.
\end{equation*}
Moreover, since the integral over $\alpha D_s$ of the remainder term in the expansion of $\frac{y_2y_3\cos\theta'}{|\mathbf{y}-\mathbf{x}|^4}$ into $\frac{\varphi'\theta'}{(\varphi'^{\mkern 2mu 2}+\theta'^{\mkern 2mu 2})^2}$ can be controlled by a constant independent of $\alpha\in(0,1]$ and $s\in(0,1/e^4]$, we obtain the following lemma.

\begin{lemma}\label{Precise velocity}
    For $\varepsilon\in(0,s_0]$ and $\alpha\in(0,1]$, if $\tilde{\omega}\in L^\infty([-\pi,\pi)\times[-\pi/2,\pi/2])$ is odd‑odd symmetric and satisfies
    \begin{equation*}
        \mathcal{X}_{\alpha\Omega_\varepsilon}\leq\tilde{\omega}\,\mathcal{X}_{\left([0,\pi)\times\left[0,\frac{\pi}{2}\right]\right)}\leq\mathcal{X}_{\left([0,\pi)\times\left[0,\frac{\pi}{2}\right]\right)},
    \end{equation*}
    then we have
    \begin{equation}\label{precise v1}
        \tilde{u}_\varphi(\alpha s,\alpha g(s))\leq0,\quad \tilde{u}_\theta(\alpha s,\alpha g(s))\geq0\quad\text{for any}\quad s\in[\varepsilon,s_0],
    \end{equation}
    \begin{equation}\label{precise v2}
        \tilde{u}_\varphi(\alpha s,\alpha g(s))<C'\alpha s,\quad \tilde{u}_\theta(\alpha s,\alpha g(s))>-C'\alpha s\quad\text{for any}\quad s\in\left[s_0,e^{-1}\right),
    \end{equation}
    and
    \begin{equation}\label{precise v3}
        \frac{1}{\sin(\alpha\varepsilon)}\sup_{s\in[0,\alpha g(\varepsilon)]}\tilde{u}_\varphi(\alpha\varepsilon,s)\leq-h(\varepsilon)|\ln\varepsilon|+C_1f(\varepsilon)+C_2+\frac{1}{\sin(\alpha e^{-1})}\inf_{s\in[0,\alpha]}\tilde{u}_\varphi(\alpha e^{-1},s).
    \end{equation}
\end{lemma}

Now consider $\varepsilon(t):[0,\infty)\to(0,s_0]$, which is decreasing and yet to be determined. Let $\tilde{\omega}(\cdot,0)\in C^\infty([-\pi,\pi)\times[-\pi/2,\pi/2])$ be an odd‑odd symmetric function satisfying
\begin{equation*}
    \mathcal{X}_{\Omega_{\varepsilon(0)}}\leq\tilde{\omega}(\cdot,0)\,\mathcal{X}_{\left([0,\pi)\times\left[0,\frac{\pi}{2}\right]\right)}\leq\mathcal{X}_{\left([0,\pi)\times\left[0,\frac{\pi}{2}\right]\right)}.
\end{equation*}
Let $\alpha(t):[0,\infty)\to(0,1]$ satisfy $\alpha(0)=1$ and
\begin{equation}\label{alpha}	 	
    \alpha'(t)=-\left(\frac{3C'}{\rho_0\cos1}+C_3\right)\alpha(t)+\frac{\alpha(t)}{\sin(\alpha(t)e^{-1})}\inf_{s\in[0,\alpha(t)]}\frac{\tilde{u}_\varphi(\alpha(t)e^{-1},s,t)}{\cos s}\quad\text{for}\quad t\in[0,\infty),
\end{equation}
where $C_3=\frac{2C_2}{\cos1}$. From $C_2>3C_1>\frac{1}{\sin(\alpha e^{-1})}\inf_{s\in[0,\alpha]}\tilde{u}_\varphi(\alpha e^{-1},s,t)$ for any $\alpha\in(0,1]$, we obtain
\begin{equation}\label{alpha'}
    \alpha'(t)\leq-\frac{3C'}{\rho_0\cos1}\,\alpha(t)\quad\text{for any}\quad t\in[0,\infty).
\end{equation}
Moreover, from \eqref{v 1}, it follows that
\begin{equation*}
    |\tilde{u}_\varphi(\alpha(t)e^{-1},s,t)|\leq\frac{2\bigl(|\ln(2\alpha(t)e^{-1})|+C_I\bigr)}{\pi e}\,\alpha(t),
\end{equation*}
where $I=\|\omega(\cdot,0)\|_{L^1(\mathbb{S}^2)}\in(0,4\pi]$, and $C_I$ coincides with that in \eqref{v 3}. Hence $\alpha(t)$ remains positive.

We now state the following claim.
\begin{claim}\label{claim:shrink}
    We can choose a suitable $\varepsilon(t):[0,\infty)\to(0,s_0]$ satisfying
\begin{equation}\label{2/pi}
    \lim_{t\to\infty}\frac{\ln(-\ln\varepsilon(t))}{t}=\frac{2}{\pi},
\end{equation}
such that the points on the boundary of $\{(\varphi',\theta'):\tilde{\omega}(\varphi',\theta',t)=1\}$ never enter $\alpha(t)\Omega_{\varepsilon(t)}$.
\end{claim}

Assuming this claim, \eqref{main1} follows immediately. Indeed, for $s\in(0,\alpha(t)g(\varepsilon(t))]$, since $\tilde{\omega}(t,0,s)=0$ by odd symmetry and $\tilde{\omega}(t,\alpha(t)\varepsilon(t),s)=1$, we have
\begin{equation*}
    \|\nabla\tilde{\omega}(\cdot,t)\|_{L^\infty}\geq\frac{1}{\alpha(t)\varepsilon(t)}\geq\frac{1}{\varepsilon(t)}.
\end{equation*}
Combined with \eqref{2/pi}, we finally arrive at
\begin{equation*}
    \liminf_{t\to\infty}\frac{\ln\ln\|\nabla\tilde{\omega}(\cdot,t)\|_{L^\infty\left([-\pi,\pi)\times(0,\pi/2]\right)}}{t}\geq\frac{2}{\pi},
\end{equation*}
which together with \eqref{main2} completes the proof.

\bigskip

It remains to prove the claim.

\textbf{Proof of the claim.}
First, from $\|\omega(\cdot,t)\|_{L^\infty(\mathbb{S}^2)}=1$, for any $\mathbf{x}\in\mathbb{S}^2$ we have
\begin{equation}\label{v bound}
\begin{split}
    |u(\mathbf{x},t)|&\leq\left|\frac{1}{2\pi}\int_{\mathbb{S}^2}\frac{\mathbf{x}\wedge\mathbf{y}}{|\mathbf{x}-\mathbf{y}|^2}\,\omega(\mathbf{y},t)\,d\sigma(\mathbf{y})\right|\\
    &\leq\frac{1}{2\pi}\int_{\mathbb{S}^2}\frac{|\mathbf{x}-\mathbf{y}|}{|\mathbf{x}-\mathbf{y}|^2}\,d\sigma(\mathbf{y})\\
    &\lesssim\frac{1}{2\pi}\int^{2\pi}_0\int^{2\pi}_0\frac{1}{r}\,dr\,d\theta\\
    &\lesssim1,
\end{split}         
\end{equation}
where in the first inequality we have used the inequality $|\mathbf{x}\wedge\mathbf{y}|\leq |\mathbf{x}-\mathbf{y}|$ from  \cite[Lemma A.1]{DR}.  

On the other hand, for any $(\varphi(t),\theta(t))\in(0,\pi)\times(0,\pi/2)$, denote
\begin{equation*}
    x(t)=(\cos\theta(t)\cos\varphi(t),\ \cos\theta(t)\sin\varphi(t),\ \sin\theta(t)),
\end{equation*}
and
\begin{equation*}
\begin{split}
    &\mathbf{e}_\varphi(\varphi(t),\theta(t))=(-\sin\varphi(t),\cos\varphi(t),0),\\
    &\mathbf{e}_\theta(\varphi(t),\theta(t))=(-\sin\theta(t)\cos\varphi(t),-\sin\theta(t)\sin\varphi(t),\cos\theta(t)).
\end{split}        
\end{equation*}
A direct computation yields
\begin{equation}\label{phi t}
    \begin{split}
        \tilde{u}_\varphi(\varphi(t),\theta(t))&=\dot{x}(t)\cdot \mathbf{e}_\varphi(\varphi(t),\theta(t))\\
        &=\left(\frac{d}{dt}\bigl(\cos\theta(t)\cos\varphi(t),\ \cos\theta(t)\sin\varphi(t),\ \sin\theta(t)\bigr)\right)\cdot \mathbf{e}_\varphi(\varphi(t),\theta(t))\\
        &=\dot{\varphi}(t)\cos\theta(t),
    \end{split}
\end{equation}
\begin{equation}\label{theta t}
    \begin{split}
        \tilde{u}_\theta(\varphi(t),\theta(t))&=\dot{x}(t)\cdot \mathbf{e}_\theta(\varphi(t),\theta(t))\\
        &=\left(\frac{d}{dt}\bigl(\cos\theta(t)\cos\varphi(t),\ \cos\theta(t)\sin\varphi(t),\ \sin\theta(t)\bigr)\right)\cdot \mathbf{e}_\theta(\varphi(t),\theta(t))\\
        &=\dot{\theta}(t).
    \end{split}
\end{equation}
Denote by $\mathbf{e}$ the unit outward normal vector on the boundary of $\alpha(t)\Omega_{\varepsilon(t)}$, and let $\mathbf{e}_1=(1,0)$, $\mathbf{e}_2=(0,1)$. From \eqref{phi t} and \eqref{theta t} we further define
\begin{equation*}
    \bar{u}(\varepsilon',\theta',t):=\frac{\tilde{u}_\varphi(\varphi',\theta',t)}{\cos\theta'}\,\mathbf{e}_1+\tilde{u}_\theta(\varphi',\theta',t)\,\mathbf{e}_2.
\end{equation*}
On $\{(\alpha(t)e^{-1},s):s\in[0,\alpha(t)]\}$, we have $\mathbf{e}=\mathbf{e}_1$. Using \eqref{alpha} and considering separately the cases
\[
\inf_{s\in[0,\alpha(t)]}\tilde{u}_\varphi(\alpha(t)e^{-1},s,t)\geq0\quad\text{and}\quad\inf_{s\in[0,\alpha(t)]}\tilde{u}_\varphi(\alpha(t)e^{-1},s,t)<0,
\]
we obtain
\begin{equation*}
    \begin{split}
        \bar{u}(\alpha(t)e^{-1},s,t)\cdot\mathbf{e}&=\frac{\tilde{u}_\varphi(\alpha(t)e^{-1},s,t)}{\cos s}\\
        &>e\alpha'(t).
    \end{split}
\end{equation*}
On $\{(\alpha(t)s,\alpha(t)g(s)):s\in[\varepsilon(t),1/e]\}$, we have
\begin{equation*}
    \mathbf{e}=-v_s\,\mathbf{e}_1+\frac{v_s}{g'(s)}\,\mathbf{e}_2,
\end{equation*}
where $v_s=(1+g'(s)^{-2})^{-1/2}$. By Lemma~\ref{Precise velocity}, we obtain
\begin{equation*}
    \bar{u}(\alpha(t)s,\alpha(t)g(s),t)\cdot\mathbf{e}\geq0\quad\text{if } s\in[\varepsilon(t),s_0],
\end{equation*}
\begin{equation*}
    \bar{u}(\alpha(t)s,\alpha(t)g(s),t)\cdot\mathbf{e}\geq-\frac{2C'}{\cos1}v_s\alpha(t)s>-\frac{3C'}{e\cos1}v_s\alpha(t)\quad\text{if } s\in\left[s_0,e^{-1}\right].
\end{equation*}
Moreover, using \eqref{g'(s)} and the definition of $\rho_0$, it follows that
\begin{equation*}
    \alpha'(t)(s,g(s))\cdot\mathbf{e}<0\quad\text{if } s\in[\varepsilon(t),s_0],
\end{equation*}
\begin{equation*}
    \alpha'(t)(s,g(s))\cdot\mathbf{e}\leq \rho_0v_s\alpha'(t)<-\frac{3C'}{e\cos1}v_s\alpha(t).
\end{equation*}
On $\{(\alpha(t)\varepsilon(t),s):s\in[0,\alpha(t)g(\varepsilon(t))]\}$, $\mathbf{e}=\mathbf{e}_1$. Using Lemma~\ref{Precise velocity} yields
\begin{equation}\label{bar u}
\begin{split}
    \bar{u}(\alpha(t)\varepsilon(t),s,t)\cdot \mathbf{e}\geq\frac{\sin\varepsilon(t)}{\varepsilon(t)}h(\varepsilon(t))&|\ln\varepsilon(t)|\,\alpha(t)\varepsilon(t)-\frac{C_1f(\varepsilon(t))+C_2}{\cos1}\,\alpha(t)\varepsilon(t)\\
    &+\frac{\sin(\alpha(t)\sin\varepsilon(t))}{\sin(\alpha(t)e^{-1})\cos s}\,\inf_{s'\in[0,\alpha(t)]}\tilde{u}_\varphi(\alpha(t)e^{-1},s',t).
\end{split}        
\end{equation}
Notice that
\begin{equation*}
    \frac{\sin(\alpha(t)\sin\varepsilon(t))}{\sin(\alpha(t)e^{-1})\cos s}=e\,\varepsilon(t)+O(\alpha^2(t))
\end{equation*}
for $s\in[0,\alpha(t)]$. Similarly,
\begin{equation*}
    \frac{1}{\sin(\alpha(t)e^{-1})\cos s}=\frac{e}{\alpha(t)}+O(\alpha(t)).
\end{equation*}
For any $s\in[0,\alpha(t)g(\varepsilon(t))]$, in order for the right‑hand side of \eqref{bar u} to be strictly greater than $-(\alpha(t)\varepsilon(t))'$, assume that $s(t)\in[0,\alpha(t)]$ satisfies
\begin{equation*}
    \tilde{u}_\varphi(\alpha(t)e^{-1},s(t),t)=\inf_{s'\in[0,\alpha(t)]}\tilde{u}_\varphi(\alpha(t)e^{-1},s',t).
\end{equation*}
Then it holds that
\begin{equation*}
    \min\left\{\tilde{u}_\varphi(\alpha(t)e^{-1},s(t),t),\ \frac{\tilde{u}_\varphi(\alpha(t)e^{-1},s(t),t)}{\cos\alpha(t)}\right\}
    \leq\inf_{s'\in[0,\alpha(t)]}\frac{\tilde{u}_\varphi(\alpha(t)e^{-1},s',t)}{\cos s'}
    \leq\frac{\tilde{u}_\varphi(\alpha(t)e^{-1},s(t),t)}{\cos s(t)}.
\end{equation*}
Combining this with the equation satisfied by $\alpha(t)$, together with \eqref{v bound} and the definition of $f(\varepsilon)$, it follows that there exists a constant $C''>0$ such that it suffices for $\varepsilon(t)$ to satisfy
\begin{equation}\label{varepsilon}
    \varepsilon'(t)>-\left(\frac{\sin\varepsilon(t)}{\varepsilon(t)}h(\varepsilon(t))|\ln\varepsilon(t)|-\frac{C_1}{\cos1}\ln\ln(e^e-1+|\ln\varepsilon(t)|)-C''\right)\varepsilon(t).
\end{equation}
Consider the following ODE:
\begin{equation*}
    \varepsilon'(t)=-\left(\frac{\sin\varepsilon(t)}{\varepsilon(t)}h(\varepsilon(t))|\ln\varepsilon(t)|-\frac{C_1}{\cos1}\ln\ln(e^e-1+|\ln\varepsilon(t)|)-(C''+1)\right)\varepsilon(t),
\end{equation*}
and take $\varepsilon(0)>0$ sufficiently small such that $\varepsilon'(0)<0$. Let $k(t):=\ln(-\ln\varepsilon(t))$. Then we have
\begin{equation*}
    k'(t)=\bar{h}\bigl(\exp(-e^{k(t)})\bigr),
\end{equation*}
where $\lim_{s\to0}\bar{h}(s)=2/\pi$.
Consequently, \eqref{varepsilon} admits a monotonically decreasing solution satisfying \eqref{2/pi}.

   \section{Rotating sphere}
In this section we briefly discuss the case of a rotating sphere. We show that Theorems~\ref{Main1} and~\ref{Main2} extend naturally to the Euler equation on a sphere that rotates at a constant angular speed.

Assume the sphere rotates about the $x_3$-axis with constant angular velocity $\Omega$. Then the two‑dimensional Euler equation on the rotating sphere takes the form
\begin{equation}\label{rotating w}
    \partial_t\omega(\mathbf{x},t)+u(\mathbf{x},t)\cdot\nabla\bigl(\omega(\mathbf{x},t)+2\Omega x_3\bigr)=0,\qquad \text{on } \mathbb{S}^2\times\mathbb{R}.
\end{equation}
Introduce the absolute vorticity
\begin{equation*}
    \zeta(\mathbf{x},t):=\omega(\mathbf{x},t)+2\Omega x_3.
\end{equation*}
In terms of $\zeta$, equation \eqref{rotating w} can be rewritten as
\begin{equation}\label{rotating zeta}
    \partial_t \zeta(\mathbf{x},t)+\nabla^\perp\bigl(\mathcal{G}\zeta(\mathbf{x},t)-\Omega x_3\bigr)\cdot\nabla\zeta(\mathbf{x},t)=0,\qquad \text{on } \mathbb{S}^2\times\mathbb{R}.
\end{equation}
Denote by $(E_\Omega)$ the above equation. When $\Omega=0$, $(E_0)$ reduces to \eqref{Eulerw} with $\zeta$ in place of $\omega$.

A direct computation yields the following relation: $\zeta(\mathbf{x},t)$ is a solution of $(E_0)$ if and only if $\zeta\bigl(R^{\mathbf{e}_3}_{\Omega t}\mathbf{x},t\bigr)$ is a solution of $(E_\Omega)$, where $\mathbf{e}_3=(0,0,1)$ and $R^{\mathbf{p}}_\alpha$ denotes the rotation by angle $\alpha$ about the unit vector $\mathbf{p}$ (given explicitly by Rodrigues' formula).

Indeed, in spherical coordinates $(\varphi,\theta)$ the rotation $R^{\mathbf{e}_3}_{\Omega t}$ simply shifts the longitude: it sends a point with coordinates $(\varphi',\theta')$ to $(\varphi'+\Omega t,\theta')$. Therefore, if we define $\zeta_\Omega(\mathbf{y},t):=\zeta\bigl(R^{\mathbf{e}_3}_{\Omega t}\mathbf{y},t\bigr)$, then for the corresponding coordinate functions $\tilde\zeta(\varphi',\theta',t)=\zeta(\mathbf{y},t)$ we have
\begin{equation*}
    \tilde\zeta_\Omega(\varphi',\theta',t)=\zeta_\Omega(\mathbf{y},t)=\zeta\bigl(R^{\mathbf{e}_3}_{\Omega t}\mathbf{y},t\bigr)=\tilde\zeta(\varphi'+\Omega t,\theta',t).
\end{equation*}
Consequently,
\begin{equation*}
    \|\nabla\tilde\zeta_\Omega(\cdot,t)\|_{L^\infty}=\|\nabla\tilde\zeta(\cdot,t)\|_{L^\infty},
\end{equation*}
because a translation in $\varphi$ does not change the $L^\infty$ norm of the gradient.

Since Theorems~\ref{Main1} and~\ref{Main2} have been proved for the non‑rotating case $(E_0)$, the above equivalence immediately implies that the same statements (with the same optimal constant $2/\pi$) hold for the rotating sphere as well. Hence the results extend to the Euler equation on a rotating sphere.

\subsection*{\large Declarations:}

 \subsection*{Funding:}

 \par
 D. Cao and J. Fan were supported by  National Key R\&D Program of China (Grant 2022YFA1005602) and NNSF of China (Grant No. 12371212). G. Qin was supported by National Key R\&D Program of China (Grant 2025YFA1018400) and NNSF of China (Grant 12471190).
 
\subsection*{Author Contribution information}
 All authors contributed equally.

 \subsection*{Conflict of interest statement} On behalf of all authors, the corresponding author states that there is no conflict of interest.

 \subsection*{Data availability statement} All data generated or analysed during this study are included in this published article  and its supplementary information files.

\bibliographystyle{abbrv}
\bibliography{1}

@article{JYZ,
      title={Superlinear gradient growth for 2{D} {E}uler equation without boundary}, 
      author={In-Jee Jeong and Yao Yao and Tao Zhou},
      year={2025},
      Journal={Preprint, arXiv: 2507.15739},
}

@article {CJ,
    AUTHOR = {Choi, Kyudong and Jeong, In-Jee},
     TITLE = {Infinite growth in vorticity gradient of compactly supported
              planar vorticity near {L}amb dipole},
   JOURNAL = {Nonlinear Anal. Real World Appl.},
  FJOURNAL = {Nonlinear Analysis. Real World Applications. An International
              Multidisciplinary Journal},
    VOLUME = {65},
      YEAR = {2022},
     PAGES = {Paper No. 103470, 20},
      ISSN = {1468-1218,1878-5719},
   MRCLASS = {35Q31 (76B47)},
  MRNUMBER = {4350517},
       DOI = {10.1016/j.nonrwa.2021.103470},
       URL = {https://doi.org/10.1016/j.nonrwa.2021.103470},
}

@article {DS1,
    AUTHOR = {Denisov, Sergey A.},
     TITLE = {Infinite superlinear growth of the gradient for the
              two-dimensional {E}uler equation},
   JOURNAL = {Discrete Contin. Dyn. Syst.},
  FJOURNAL = {Discrete and Continuous Dynamical Systems},
    VOLUME = {23},
      YEAR = {2009},
    NUMBER = {3},
     PAGES = {755--764},
      ISSN = {1078-0947,1553-5231},
   MRCLASS = {35Q35 (35B45 35L65 76B99 76F99)},
  MRNUMBER = {2461825},
MRREVIEWER = {Chao\ Cheng\ Huang},
       DOI = {10.3934/dcds.2009.23.755},
       URL = {https://doi.org/10.3934/dcds.2009.23.755},
}

@article {DS2,
    AUTHOR = {Denisov, Sergey A.},
     TITLE = {Double exponential growth of the vorticity gradient for the
              two-dimensional {E}uler equation},
   JOURNAL = {Proc. Amer. Math. Soc.},
  FJOURNAL = {Proceedings of the American Mathematical Society},
    VOLUME = {143},
      YEAR = {2015},
    NUMBER = {3},
     PAGES = {1199--1210},
      ISSN = {0002-9939,1088-6826},
   MRCLASS = {35Q31 (35B45 76B47)},
  MRNUMBER = {3293735},
MRREVIEWER = {Francesco\ Fanelli},
       DOI = {10.1090/S0002-9939-2014-12286-6},
       URL = {https://doi.org/10.1090/S0002-9939-2014-12286-6},
}

@book {MB,
    AUTHOR = {Majda, Andrew J. and Bertozzi, Andrea L.},
     TITLE = {Vorticity and {I}ncompressible {F}low},
    SERIES = {Cambridge Texts in Applied Mathematics},
    VOLUME = {27},
 PUBLISHER = {Cambridge University Press, Cambridge},
      YEAR = {2002},
     PAGES = {xii+545},
      ISBN = {0-521-63057-6; 0-521-63948-4},
   MRCLASS = {76-02 (35Q30 35Q35 76B03 76D03 76D05)},
  MRNUMBER = {1867882},
MRREVIEWER = {Yuxi\ Zheng},
}

@article{HE,
author = {Hölder, E.},
journal = {Mathematische Zeitschrift},
pages = {727-738},
title = {Über die unbeschränkte Fortsetzbarkeit einer stetigen ebenen Bewegung in einer unbegrenzten inkompressiblen Flüssigkeit.},
url = {http://eudml.org/doc/168483},
volume = {37},
year = {1933},
}

@article {JV,
    AUTHOR = {Yudovich, V. I.},
     TITLE = {The loss of smoothness of the solutions of {E}uler equations
              with time},
   JOURNAL = {Dinamika Splo\v sn. Sredy},
  FJOURNAL = {Institut Gidrodinamiki Sibirskogo Otdelenija Akademii Nauk
              SSSR. Dinamika Splo\v sno\u i\ Sredy},
      YEAR = {1974},
    NUMBER = {16},
     PAGES = {71--78, 121},
      ISSN = {0420-0497},
   MRCLASS = {35Q99 (76.35)},
  MRNUMBER = {454419},
}

@article {KV,
    AUTHOR = {Kiselev, Alexander and \v{S}ver\'ak, Vladimir},
     TITLE = {Small scale creation for solutions of the incompressible
              two-dimensional {E}uler equation},
   JOURNAL = {Ann. of Math. (2)},
  FJOURNAL = {Annals of Mathematics. Second Series},
    VOLUME = {180},
      YEAR = {2014},
    NUMBER = {3},
     PAGES = {1205--1220},
      ISSN = {0003-486X,1939-8980},
   MRCLASS = {35Q31 (35B45 76B03)},
  MRNUMBER = {3245016},
MRREVIEWER = {Paolo\ Secchi},
       DOI = {10.4007/annals.2014.180.3.9},
       URL = {https://doi.org/10.4007/annals.2014.180.3.9},
}

@article {LH,
    AUTHOR = {Luo, Guo and Hou, Thomas Y.},
     TITLE = {Toward the finite-time blowup of the 3{D} axisymmetric {E}uler
              equations: a numerical investigation},
   JOURNAL = {Multiscale Model. Simul.},
  FJOURNAL = {Multiscale Modeling \& Simulation. A SIAM Interdisciplinary
              Journal},
    VOLUME = {12},
      YEAR = {2014},
    NUMBER = {4},
     PAGES = {1722--1776},
      ISSN = {1540-3459,1540-3467},
   MRCLASS = {35Q31 (35B44 65M20 65M60 76B03)},
  MRNUMBER = {3278833},
       DOI = {10.1137/140966411},
       URL = {https://doi.org/10.1137/140966411},
}

@article {NS,
    AUTHOR = {Nadirashvili, N. S.},
     TITLE = {Wandering solutions of the two-dimensional {E}uler equation},
   JOURNAL = {Funktsional. Anal. i Prilozhen.},
  FJOURNAL = {Akademiya Nauk SSSR. Funktsional\cprime ny\u i\ Analiz i ego
              Prilozheniya},
    VOLUME = {25},
      YEAR = {1991},
    NUMBER = {3},
     PAGES = {70--71},
      ISSN = {0374-1990},
   MRCLASS = {35Q30},
  MRNUMBER = {1139875},
       DOI = {10.1007/BF01085491},
       URL = {https://doi.org/10.1007/BF01085491},
}

@article {WW,
    AUTHOR = {Wolibner, W.},
     TITLE = {Un theor\`eme sur l'existence du mouvement plan d'un fluide
              parfait, homog\`ene, incompressible, pendant un temps
              infiniment long},
   JOURNAL = {Math. Z.},
  FJOURNAL = {Mathematische Zeitschrift},
    VOLUME = {37},
      YEAR = {1933},
    NUMBER = {1},
     PAGES = {698--726},
      ISSN = {0025-5874,1432-1823},
   MRCLASS = {99-04},
  MRNUMBER = {1545430},
       DOI = {10.1007/BF01474610},
       URL = {https://doi.org/10.1007/BF01474610},
}

@article {XX,
    AUTHOR = {Xu, Xiaoqian},
     TITLE = {Fast growth of the vorticity gradient in symmetric smooth
              domains for 2{D} incompressible ideal flow},
   JOURNAL = {J. Math. Anal. Appl.},
  FJOURNAL = {Journal of Mathematical Analysis and Applications},
    VOLUME = {439},
      YEAR = {2016},
    NUMBER = {2},
     PAGES = {594--607},
      ISSN = {0022-247X,1096-0813},
   MRCLASS = {35Q31 (76B03 76B47)},
  MRNUMBER = {3475939},
MRREVIEWER = {Franck\ Sueur},
       DOI = {10.1016/j.jmaa.2016.02.066},
       URL = {https://doi.org/10.1016/j.jmaa.2016.02.066},
}

@article{YV1,
 author = {Yudovich, V. I.},
 title = {The flow of a perfect, incompressible liquid through a given region},
 fjournal = {Soviet Physics. Doklady},
 journal = {Sov. Phys., Dokl.},
 issn = {0038-5689},
 volume = {7},
 pages = {789--791},
 year = {1962},
 language = {English},
 zbMATH = {3226068},
 Zbl = {0139.20502}
}

@article {YV2,
    AUTHOR = {Yudovich, V. I.},
     TITLE = {On the loss of smoothness of the solutions of the {E}uler
              equations and the inherent instability of flows of an ideal
              fluid},
   JOURNAL = {Chaos},
  FJOURNAL = {Chaos. An Interdisciplinary Journal of Nonlinear Science},
    VOLUME = {10},
      YEAR = {2000},
    NUMBER = {3},
     PAGES = {705--719},
      ISSN = {1054-1500,1089-7682},
   MRCLASS = {76B03 (35Q35 37N10 76E99)},
  MRNUMBER = {1791984},
       DOI = {10.1063/1.1287066},
       URL = {https://doi.org/10.1063/1.1287066},
}

@article {YV3,
    AUTHOR = {Yudovich, V. I.},
     TITLE = {Non-stationary flows of an ideal incompressible fluid},
   JOURNAL = {\v Z. Vy\v cisl. Mat i Mat. Fiz.},
  FJOURNAL = {\v Zurnal Vy\v cislitel\cprime no\u i\ Matematiki i Matemati\v
              cesko\u i\ Fiziki},
    VOLUME = {3},
      YEAR = {1963},
     PAGES = {1032--1066},
      ISSN = {0044-4669},
   MRCLASS = {35.79},
  MRNUMBER = {158189},
MRREVIEWER = {P.\ C.\ Fife},
}

@article {ZA1,
    AUTHOR = {Zlato\v{s}, Andrej},
     TITLE = {Exponential growth of the vorticity gradient for the {E}uler
              equation on the torus},
   JOURNAL = {Adv. Math.},
  FJOURNAL = {Advances in Mathematics},
    VOLUME = {268},
      YEAR = {2015},
     PAGES = {396--403},
      ISSN = {0001-8708,1090-2082},
   MRCLASS = {35Q31 (58J99)},
  MRNUMBER = {3276599},
MRREVIEWER = {Michele\ Coti Zelati},
       DOI = {10.1016/j.aim.2014.08.012},
       URL = {https://doi.org/10.1016/j.aim.2014.08.012},
}

@article {DSA,
    AUTHOR = {Denisov, Sergey A.},
     TITLE = {The sharp corner formation in 2{D} {E}uler dynamics of
              patches: infinite double exponential rate of merging},
   JOURNAL = {Arch. Ration. Mech. Anal.},
  FJOURNAL = {Archive for Rational Mechanics and Analysis},
    VOLUME = {215},
      YEAR = {2015},
    NUMBER = {2},
     PAGES = {675--705},
      ISSN = {0003-9527,1432-0673},
   MRCLASS = {35Q31 (76N10)},
  MRNUMBER = {3294414},
MRREVIEWER = {Xinyu\ He},
       DOI = {10.1007/s00205-014-0793-2},
       URL = {https://doi.org/10.1007/s00205-014-0793-2},
}

@article {KAZA,
    AUTHOR = {Kiselev, Alexander and Zlato\v{s}, Andrej},
     TITLE = {Blow up for the 2{D} {E}uler equation on some bounded domains},
   JOURNAL = {J. Differential Equations},
  FJOURNAL = {Journal of Differential Equations},
    VOLUME = {259},
      YEAR = {2015},
    NUMBER = {7},
     PAGES = {3490--3494},
      ISSN = {0022-0396,1090-2732},
   MRCLASS = {35Q31 (35B44)},
  MRNUMBER = {3360679},
       DOI = {10.1016/j.jde.2015.04.027},
       URL = {https://doi.org/10.1016/j.jde.2015.04.027},
}

@article {ZA2,
    AUTHOR = {Zlato\v{s}, Andrej},
     TITLE = {Maximal double-exponential growth for the {E}uler equation on
              the half-plane},
   JOURNAL = {Invent. Math.},
  FJOURNAL = {Inventiones Mathematicae},
    VOLUME = {243},
      YEAR = {2026},
    NUMBER = {1},
     PAGES = {117--126},
      ISSN = {0020-9910,1432-1297},
   MRCLASS = {99-06},
  MRNUMBER = {5008147},
       DOI = {10.1007/s00222-025-01374-5},
       URL = {https://doi.org/10.1007/s00222-025-01374-5},
}

@article {DE,
    AUTHOR = {Drivas, Theodore D. and Elgindi, Tarek M.},
     TITLE = {Singularity formation in the incompressible {E}uler equation
              in finite and infinite time},
   JOURNAL = {EMS Surv. Math. Sci.},
  FJOURNAL = {EMS Surveys in Mathematical Sciences},
    VOLUME = {10},
      YEAR = {2023},
    NUMBER = {1},
     PAGES = {1--100},
      ISSN = {2308-2151,2308-216X},
   MRCLASS = {35Q31 (35Q30)},
  MRNUMBER = {4667415},
       DOI = {10.4171/emss/66},
       URL = {https://doi.org/10.4171/emss/66},
}

@article {DEJ,
    AUTHOR = {Drivas, Theodore D. and Elgindi, Tarek M. and Jeong, In-Jee},
     TITLE = {Twisting in {H}amiltonian flows and perfect fluids},
   JOURNAL = {Invent. Math.},
  FJOURNAL = {Inventiones Mathematicae},
    VOLUME = {238},
      YEAR = {2024},
    NUMBER = {1},
     PAGES = {331--370},
      ISSN = {0020-9910,1432-1297},
   MRCLASS = {35Q31 (37J25 76E05)},
  MRNUMBER = {4794597},
MRREVIEWER = {Xianyi\ Zeng},
       DOI = {10.1007/s00222-024-01285-x},
       URL = {https://doi.org/10.1007/s00222-024-01285-x},
}

@inproceedings {KA,
    AUTHOR = {Kiselev, Alexander},
     TITLE = {Small scales and singularity formation in fluid dynamics},
 BOOKTITLE = {Proceedings of the {I}nternational {C}ongress of
              {M}athematicians---{R}io de {J}aneiro 2018. {V}ol. {III}.
              {I}nvited lectures},
     PAGES = {2363--2390},
 PUBLISHER = {World Sci. Publ., Hackensack, NJ},
      YEAR = {2018},
      ISBN = {978-981-3272-92-7; 978-981-3272-87-3},
   MRCLASS = {35Q35 (76B03 86A08)},
  MRNUMBER = {3966854},
MRREVIEWER = {Alberto\ Valli},
}

@incollection {KAA,
    AUTHOR = {Kiselev, Alexander A.},
     TITLE = {Small scale creation in active scalars},
 BOOKTITLE = {Progress in mathematical fluid dynamics},
    SERIES = {Lecture Notes in Math.},
    VOLUME = {2272},
     PAGES = {125--161},
 PUBLISHER = {Springer, Cham},
      YEAR = {[2020] \copyright 2020},
      ISBN = {978-3-030-54899-5; 978-3-030-54898-8},
   MRCLASS = {76B03 (35Q35)},
  MRNUMBER = {4176667},
       DOI = {10.1007/978-3-030-54899-5\_4},
       URL = {https://doi.org/10.1007/978-3-030-54899-5_4},
}

@article {DB,
    AUTHOR = {Dritschel, D. G. and Boatto, S.},
     TITLE = {The motion of point vortices on closed surfaces},
   JOURNAL = {Proc. R. Soc. A},
  FJOURNAL = {Proceedings A},
    VOLUME = {471},
      YEAR = {2015},
    NUMBER = {2176},
     PAGES = {20140890},
      ISSN = {1364-5021,1471-2946},
   MRCLASS = {76B47 (37K45)},
  MRNUMBER = {3325200},
MRREVIEWER = {Yasuhide\ Fukumoto},
       DOI = {10.1098/rspa.2014.0890},
       URL = {https://doi.org/10.1098/rspa.2014.0890},
}

@book {MCP,
    AUTHOR = {Marchioro, Carlo and Pulvirenti, Mario},
     TITLE = {Mathematical theory of incompressible nonviscous fluids},
    SERIES = {Applied Mathematical Sciences},
    VOLUME = {96},
 PUBLISHER = {Springer-Verlag, New York},
      YEAR = {1994},
     PAGES = {xii+283},
      ISBN = {0-387-94044-8},
   MRCLASS = {76-02 (35Q35 76Cxx 76E99 76F99 76M25)},
  MRNUMBER = {1245492},
MRREVIEWER = {J.\ Thomas\ Beale},
       DOI = {10.1007/978-1-4612-4284-0},
       URL = {https://doi.org/10.1007/978-1-4612-4284-0},
}

@article {TM,
    AUTHOR = {Taylor, Michael},
     TITLE = {Euler equation on a rotating surface},
   JOURNAL = {J. Funct. Anal.},
  FJOURNAL = {Journal of Functional Analysis},
    VOLUME = {270},
      YEAR = {2016},
    NUMBER = {10},
     PAGES = {3884--3945},
      ISSN = {0022-1236,1096-0783},
   MRCLASS = {35Q31 (35B35 35R09)},
  MRNUMBER = {3478875},
MRREVIEWER = {Roger\ Lewandowski},
       DOI = {10.1016/j.jfa.2016.02.023},
       URL = {https://doi.org/10.1016/j.jfa.2016.02.023},
}

@article {CAG,
    AUTHOR = {Constantin, A. and Germain, P.},
     TITLE = {Stratospheric planetary flows from the perspective of the
              {E}uler equation on a rotating sphere},
   JOURNAL = {Arch. Ration. Mech. Anal.},
  FJOURNAL = {Archive for Rational Mechanics and Analysis},
    VOLUME = {245},
      YEAR = {2022},
    NUMBER = {1},
     PAGES = {587--644},
      ISSN = {0003-9527,1432-0673},
   MRCLASS = {35Q31 (76U60)},
  MRNUMBER = {4444081},
MRREVIEWER = {Yang\ Pu},
       DOI = {10.1007/s00205-022-01791-3},
       URL = {https://doi.org/10.1007/s00205-022-01791-3},
}

@article {CLW,
    AUTHOR = {Cao, Daomin and Li, Shuanglong and Wang, Guodong},
     TITLE = {Desingularization of vortices for the incompressible {E}uler
              equation on a sphere},
   JOURNAL = {Calc. Var. Partial Differential Equations},
  FJOURNAL = {Calculus of Variations and Partial Differential Equations},
    VOLUME = {64},
      YEAR = {2025},
    NUMBER = {9},
     PAGES = {Paper No. 279, 26},
      ISSN = {0944-2669,1432-0835},
   MRCLASS = {76B47 (35Q31 76B03)},
  MRNUMBER = {4970251},
       DOI = {10.1007/s00526-025-03154-8},
       URL = {https://doi.org/10.1007/s00526-025-03154-8},
}

@article{DR,
      title={Confinement results near point vortices on the rotating sphere}, 
      author={Martin Donati and Emeric Roulley},
      year={2026},
      Journal={Preprint, arXiv: 2602.10061},
}

@article {HLY,
    AUTHOR = {Hu, Zhongtian and Luo, Chenyun and Yao, Yao},
     TITLE = {Small scale creation for 2{D} free boundary {E}uler equations
              with surface tension},
   JOURNAL = {Ann. PDE},
  FJOURNAL = {Annals of PDE. Journal Dedicated to the Analysis of Problems
              from Physical Sciences},
    VOLUME = {10},
      YEAR = {2024},
    NUMBER = {2},
     PAGES = {Paper No. 13, 19},
      ISSN = {2524-5317,2199-2576},
   MRCLASS = {35Q31 (35R35 76B45)},
  MRNUMBER = {4769136},
MRREVIEWER = {Qin\ Zhao},
       DOI = {10.1007/s40818-024-00179-8},
       URL = {https://doi.org/10.1007/s40818-024-00179-8},
}

@article{CGLZ,
      title={The onset of instability for zonal stratospheric flows}, 
      author={Constantin,Adrian  and  Germain,Pierre and Lin, Zhiwu and Zhu,Hao },
      year={2025},
      journal={arXiv:2503.14191}
}

@article{CW24,
      title={On {A}rnold-type stability theorems for the {E}uler equation on a sphere}, 
      author={Daomin Cao and Guodong Wang},
      year={2024},
      journal={arXiv: 2407.06752}
}

@article {CSMC,
    AUTHOR = {Caprino, S. and Marchioro, C.},
     TITLE = {On nonlinear stability of stationary {E}uler flows on a
              rotating sphere},
   JOURNAL = {J. Math. Anal. Appl.},
  FJOURNAL = {Journal of Mathematical Analysis and Applications},
    VOLUME = {129},
      YEAR = {1988},
    NUMBER = {1},
     PAGES = {24--36},
      ISSN = {0022-247X},
   MRCLASS = {76E30 (58D25)},
  MRNUMBER = {921375},
MRREVIEWER = {D.\ A.\ Nield},
       DOI = {10.1016/0022-247X(88)90231-4},
       URL = {https://doi.org/10.1016/0022-247X(88)90231-4},
}

@article {GHR,
    AUTHOR = {Garc\'ia, Claudia and Hassainia, Zineb and Roulley, Emeric},
     TITLE = {Dynamics of vortex cap solutions on the rotating unit sphere},
   JOURNAL = {J. Differential Equations},
  FJOURNAL = {Journal of Differential Equations},
    VOLUME = {417},
      YEAR = {2025},
     PAGES = {1--63},
      ISSN = {0022-0396,1090-2732},
   MRCLASS = {35Q31 (76U05)},
  MRNUMBER = {4827843},
MRREVIEWER = {Olga\ S.\ Rozanova},
       DOI = {10.1016/j.jde.2024.11.012},
       URL = {https://doi.org/10.1016/j.jde.2024.11.012},
}

@article{GE,
      title={Filamentation near monotone zonal vortex caps}, 
      author={G. M. Marin and E. Roulley},
      year={2025},
      journal={arXiv:2505.12197}
}

@article {SZ1,
    AUTHOR = {Sakajo, Takashi and Zou, Changjun},
     TITLE = {Regularization for point vortices on {$\Bbb S^2$}},
   JOURNAL = {Nonlinearity},
  FJOURNAL = {Nonlinearity},
    VOLUME = {38},
      YEAR = {2025},
    NUMBER = {11},
     PAGES = {Paper No. 115021, 36},
      ISSN = {0951-7715,1361-6544},
   MRCLASS = {76B03 (76B47)},
  MRNUMBER = {4994622},
       DOI = {10.1088/1361-6544/ae1ee0},
       URL = {https://doi.org/10.1088/1361-6544/ae1ee0},
}

@article{SZ2,
      title={{$C^1$} type regularization for point vortices on {$\Bbb S^2$}}, 
      author={Sakajo, Takashi and Zou, Changjun},
      year={2024},
      journal={arXiv:2411.15176}
}

@article {LF,
    AUTHOR = {Laurent-Polz, F.},
     TITLE = {Point vortices on a rotating sphere},
   JOURNAL = {Regul. Chaotic Dyn.},
  FJOURNAL = {Regular \& Chaotic Dynamics. International Scientific Journal},
    VOLUME = {10},
      YEAR = {2005},
    NUMBER = {1},
     PAGES = {39--58},
      ISSN = {1560-3547,1468-4845},
   MRCLASS = {76B47 (37J15 37N10 86A10)},
  MRNUMBER = {2136829},
MRREVIEWER = {Koji\ Ohkitani},
       DOI = {10.1070/RD2005v010n01ABEH000299},
       URL = {https://doi.org/10.1070/RD2005v010n01ABEH000299},
}

\end{document}